\documentclass[a4paper,11pt]{amsart}

\usepackage{hyperref}
\hypersetup{backref,colorlinks=true,allcolors=black}
\usepackage[foot]{amsaddr}
\usepackage{mathrsfs}
\usepackage{enumerate}
\usepackage{dsfont}

\newtheorem{prop}{Proposition}
\newtheorem{lemma}{Lemma}
\newtheorem{definition}{Definition}

\newtheorem{theorem}{Theorem}
\newtheorem{remark}{Remark}

\newtheorem{innercustomthm}{Theorem}
\newenvironment{customthm}[1]
{\renewcommand\theinnercustomthm{#1}\innercustomthm}
{\endinnercustomthm}

\newcommand{\ignore}[1]{}
\newcommand{\R}{\mathbb{R}}
\newcommand{\N}{\mathbb{N}}
\newcommand{\E}{\mathbb{E}}
\newcommand{\X}{\mathcal{X}}
\newcommand{\Y}{\mathcal{Y}}
\newcommand{\W}{\mathcal{W}}
\newcommand{\Z}{\mathcal{Z}}
\newcommand{\I}{\mathcal{I}}
\newcommand{\on}[1]{\operatorname{#1}}
\newcommand{\norm}[1]{\left\lVert #1 \right\rVert}
\newcommand{\abs}[1]{\left\vert #1 \right\rvert}

\begin{document}
\bibliographystyle{amsalpha}


\author{Solesne Bourguin$^1$}
\address{$^1$Boston University, Department of Mathematics and Statistics,  111 Cummington Mall, Boston, MA 02215, USA}
\email{bourguin@math.bu.edu}
\author{Charles-Philippe Diez$^2$}
\address{$^2$CNRS, Universit\'e de Lille, Laboratoire Paul Painlev\'e,
  UMR 8524, F-59655 Villeneuve d'Ascq, France.}
\email{charles-philippe.diez@univ-lille.fr}
\author{Ciprian A. Tudor$^2$}
\email{ciprian.tudor@math.univ-lille.fr}

\title[Large correlated Wishart matrices with
  chaotic entries]{Limiting behavior of large correlated Wishart matrices with
    chaotic entries}
  
\thanks{S. Bourguin was supported in part by the Simons Foundation
  grant 635136. C. Tudor was supported in part by the Labex CEMPI (ANR-11-LABX-0007-01) and MATHAMSUD project SARC (19-MATH-06)}

\begin{abstract}
We study the fluctuations, as $d,n\to \infty$,  of the Wishart matrix
$\W_{n,d}= \frac{1}{d} \X_{n,d} \X_{n,d} ^ {T} $ associated to a
$n\times d$ random matrix $\X_{n,d}$ with non-Gaussian entries. We
analyze the limiting behavior in distribution of $\W _{n,d}$ in two
situations: when the entries of $\X_{n,d}$ are independent elements
of a Wiener chaos of arbitrary order and when the entries are
partially correlated and belong to the second Wiener chaos. In the
first case, we show that the (suitably normalized) Wishart matrix
converges in distribution to a Gaussian matrix while in the correlated
case, we obtain its convergence in law to a diagonal non-Gaussian
matrix. In both cases, we derive the rate of convergence in the
Wasserstein distance via Malliavin calculus and analysis on Wiener
space.
\end{abstract}

\subjclass[2010]{60B20, 60F05, 60H07, 60G22}
\keywords{Wishart matrix; multiple stochastic integrals; Malliavin
  calculus; Stein's method; Rosenblatt process; fractional Brownian
  motion; high-dimensional regime}

\maketitle


\section{Introduction}
\noindent Random matrix theory plays an important role in various
areas of applications, including statistical physics, engineering
sciences, signal processing or mathematical finance. The various tools
that can be used to study random matrices come from different branches
of mathematics, such as combinatorics, non-commutative algebra,
geometry, spectral analysis and, of course, probability and
statistics. We focus on a special type of random matrices, called {\it
  Wishart matrices}, which have been introduced in
\cite{wishart_generalised_1928}. Given a $n\times d$ random matrix
$\X_{n,d}= (X_{ij})_{1\leq i\leq n,\ 1\leq j\leq d}$ with real
entries, its associated Wishart matrix $\W_{n,d}= (W_{ij})_{1\leq i,j
  \leq n}$ is the symmetric $n\times n$ matrix $\W _{n,d}= \frac{1}{d}
\X _{n,d} \X_{n,d}^ {T}$ ($\X^{T} $ being the transpose  of the matrix $\X$). The class of  Wishart matrices constitutes a
special class of sample covariance matrices with applications in
multivariate analysis or statistical theory, see e.g., the surveys
\cite{bishop_introduction_2018,johnstone_high_2007,rasmussen_gaussian_2006}. The
limiting behavior of this type of random matrices, as $d$ goes to infinity and $n$ is fixed (which is
referred to as the classical or finite dimensional regime) or when both $n, d$ tend to
infinity (usually called the high dimensional regime), has been
studied by many authors. The starting point of this analysis is the
situation where the entries of the matrix $\X_{n,d}$ are i.i.d. and
$n$ is fixed. In this case, the Wishart matrix associated to
$\X_{n,d}$ converges almost surely, as $d\to \infty$, to the $n\times
n$ identity matrix $\I_{n}$ by the strong law of large numbers and the
renormalized Wishart matrix $\sqrt{d} (\W_{n,d}- \I_{n}) $ satisfies a
Central Limit Theorem (CLT in the sequel). Later, due to the
increasing need of handling large data sets, several authors
investigated the high dimensional regime, when the matrix size $n$
also goes to infinity. Different strategies have been considered in
this case. A classical approach is based on the study of the empirical
spectral distribution and of the  eigenvalues of $\W_{n,d}$. It is
well known that if $n,d\to \infty$ such that $n/d\to c\in (0,
\infty)$, then the empirical spectral distribution of the Wishart
matrix converges weakly to the so-called Marchenko-Pastur distribution
(see \cite{marcenko_distribution_1967}).   A more recent approach
consists in analyzing the distance in distribution (for example, under the total
variation distance or Wasserstein distance) between the renormalized
Wishart matrix $\sqrt{d} (\W_{n,d}- \I_{n}) $ and its limiting
distribution when $d$ and $n$ are large. This approach has been used
in, e.g.,
\cite{bubeck_testing_2016,bubeck_entropic_2018,jiang_approximation_2015,racz_smooth_2019,nourdin_asymptotic_2018}. It
has been discovered that the distance (in the Wasserstein or total
variation sense) between the distribution of the renormalized Wishart
matrix and its limiting distribution (when this limit is Gaussian,
which happens in all the cases except when the entries have a strong
enough correlation, see \cite{nourdin_asymptotic_2018}), as
$n,d\to\infty$,  is of order less than $n^3/d$.  In the
above references, several situations have been studied: the entries of
the initial matrix $\X _{n, d}$ are independent and Gaussian (see
\cite{bubeck_testing_2016,jiang_approximation_2015,racz_smooth_2019}),
the entries are independent and not necessarily Gaussian (they are
supposed to have a log-concave distribution in
\cite{bubeck_entropic_2018}) or the entries are Gaussian and partially
correlated (see \cite{nourdin_asymptotic_2018}).  While in most
references the proofs are based  on entropy or moments analysis, in
\cite{nourdin_asymptotic_2018} the authors use the recent
Stein-Malliavin calculus (see \cite{nourdin_normal_2012}).
\\~\\
Our purpose is to use the techniques of Malliavin calculus and
analysis on Wiener space in order to generalize the above results in
two directions. First, we start with an $n\times d$ matrix $\X _{n,
  d}$ whose entries are independent  (not necessarily identically
distributed) elements of Wiener chaoses of arbitrary order. That is, we assume that for every
$1 \leq i \leq n$ and for every $1 \leq j \leq d$, 
\begin{equation}
\label{xij1-intro}
X_{ij}= I_{q_i} (f_{ij}),
\end{equation}
with $f_{ij}\in \mathfrak{H}^{\odot q_{i}  }$, where $q_{i}\geq 1$ and
the maximum of the $q_{i}$'s is bounded by an integer number $N_{0}$  for
every $1 \leq i \leq n$. In \eqref{xij1-intro}, $I_{q}$ denotes the
multiple Wiener integral of order $q$ with respect to an isonormal process $W$. Assume that the entries  have the same second and fourth moments, i.e., for every $1 \leq i \leq n$ and $1 \leq j \leq d$,

\begin{equation*}
  \E \left(  X^ {2}_{ij}  \right)= q_{i}!
  \norm{f_{ij}}_{\mathfrak{H}^{\otimes q_i}}^2 =1 \quad \mbox {and} \quad
  \E \left(  X^ {4}_{ij}  \right) = m_{4}. 
\end{equation*}
In this
situation we obtain the convergence in law of the corresponding
renormalized Wishart matrix $\widetilde{\W}_{n, d}= (\widetilde{W}_{ij})_{1\leq i,j\leq n} $ with entries $\widetilde{W}_{ij}= \sqrt{d} W_{ij}$ for $1\leq i,j\leq n$  to the GOE (Gaussian Orthogonal Ensemble)
matrix $\Z _{n}$ given by \eqref{goe}. This is a symmetric random
matrix $\Z_{n}= (Z_{ij}) _{1\leq i,j\leq n}$ whose diagonal elements
follow the distribution $Z_{ii}\sim N(0, m_{4}-1)$ while the
non-diagonal entries are such that $Z_{ij}\sim N(0,1)$ if $1\leq i<j
\leq n$ and  $Z_{ij}=Z_{ji}$ if $1\leq i<j\leq n$,  the variables
$\{Z_{ij} \colon i\leq j\}$ being independent.

The study of Wishart matrices  based on an initial matrix $\X_{n,d}$
with independent elements in (potentially different) Wiener chaoses 
is motivated by the following facts.  As mentioned above, Wishart
matrices can be viewed as sample covariance matrices and the elements
of the matrix $\X_{n,d}$ can be interpreted as the data. In recent
years, the statistical inference based on observations belonging to
Wiener chaoses of arbitrary order has been intensively studied (see,
among others, \cite{chronopoulou_self-similarity_2011, clausel_asymptotic_2014,
  pipiras_long-range_2017, tudor_analysis_2013}.  Another motivation
is related to the concept of universality, which has been tremendously
studied for random matrices by many authors (see e.g. \cite{edelman_beyond_2016}
and the references therein). Loosely speaking, the notion of
universality implies to understand the behavior of random matrices
with entries from a general (non necessarily Gaussian) distribution
and to see if the behavior displayed by Gaussian matrices still holds in the general case. 

We actually show that, when $n,d\to
\infty$, the distance between the renormalized Wishart matrix $\widetilde{\W}_{n, d}= (\widetilde{W}_{ij})_{1\leq i,j\leq n} $ and the
GOE matrix is of order less that $n^3/d$. This generalizes
the results of \cite{bubeck_testing_2016,bubeck_entropic_2018,jiang_approximation_2015,racz_smooth_2019}. More precisely, we prove the following result.
\begin{theorem}
  \label{mainresult-independence}
Consider the renormalized  Wishart matrix $\widetilde{\W}_{n, d}$ with
entries given by \eqref{hatwij}. Then for every $n\geq 1$,
$\widetilde{\W}_{n, d}$  converges in distribution componentwise, as $d\to \infty$, to
the matrix $\Z_n$ given by \eqref{goe}. Moreover, there exists a positive constant $C$ such that for every $n,d\geq
1$, 
\begin{equation}\label{28i-10}
d_{W} (\widetilde{\W}_{n, d}, \Z_n ) \leq C\sqrt{ \frac{n ^ {3}}{d}},
\end{equation}
where $d_W$ denotes the Wasserstein distance defined in Section \ref{secdistances}.
\end{theorem}
Another direction of study is to start with a matrix $\X_{n,d}$ whose elements
are non-Gaussian and partially correlated. As pointed out in
e.g. \cite{bubeck_entropic_2018}, obtaining an approximation result
without the assumption of independence represents a natural question
which has been a subject of wide interest. We will assume that these
entries are elements of the second Wiener chaos, correlated on the
same row, with the correlation being given by the increments of the
Rosenblatt process (see Section \ref{sec4} for the definition and
basic properties of this stochastic process). More precisely, the
entries of the matrix   $\X_{n,d}= (X_{ij})_{1\leq i\leq n, 1\leq
  j\leq }$ are given by $X_{ij}= Z ^{H, i}_{j}- Z^{H, i} _{j-1}$,
where $Z^{H, i}$, $1\leq i\leq n$ are $n$-independent Rosenblatt
processes with the same Hurst parameter $H\in \left( \frac{1}{2},
  1\right)$.  The definition and basic properties of the Rosenblatt
process are recalled in Section \ref{sec4}. This stochastic process is
a non-Gaussian self-similar process with stationary increments and
long-memory. Due to these properties, it found several applications in
various areas (hydrology, finance, interned traffic analysis, and more).  For more details on the theoretical aspects  and practical applications of the Rosenblatt process, we refer to the  monographs \cite{pipiras_long-range_2017, tudor_analysis_2013}.

Note that the
correlation structure of the Rosenblatt process is the same as the one
of the fractional Brownian motion (fBm). In this sense, the
correlation on the rows of the matrix $\X_{n,d}$ considered in our
work is the same as in \cite{nourdin_asymptotic_2018} (where the
entries are increments of the fBm). Nevertheless, the non-Gaussian
character of the entries  brings  more complexity and  leads to a
different behavior  of the associated Wishart matrix. Actually, we
show that the renormalized Wishart matrix $\widetilde{\W}_{n,d} =(\widetilde{W}_{ij})_{1\leq i,j\leq n} $ with $\widetilde{W}_{ij}= c_{1,H}^{-1} d ^{1-H}W_{ij} $  (the constant $c_{1,H}$ is defined in  \eqref{c1h}) converges  to a diagonal
matrix whose diagonal entries are random variables distributed
according to the Rosenblatt distribution and we are also able to
quantify the distance  associated to this limit theorem. Our result
can be stated as follows.
\begin{theorem}
 \label{mainresult-correlated}
Let $\widetilde{\W}_{n, d}$ be the renormalized Wishart matrix
\eqref{twij} and let $\mathcal{R}^ {H}_{n}$ be the diagonal matrix with
entries given by \eqref{29i-1}. Then, for every $n\geq 1$, the random  matrix  $\widetilde{\W}_{n, d}$ converges componentwise in distribution, as $d\to \infty$, to the matrix $\mathcal{R}^ {H}_{n}$. Moreover, there exists a positive constant $C$ such that as  $n,d\geq
1$, 
\begin{equation*}
d_{W}\left( \widetilde{\W}_{n, d}, \mathcal{R}^ {H}_{n}\right) \leq
C \begin{cases} nd ^ {\frac{1}{2}-H} & \mbox{if } H \in\left( \frac{1}{2}, \frac{3}{4}\right)\\
n\sqrt{\log(d)} d ^ {-\frac{1}{4}} & \mbox {if } H=\frac{3}{4}\\
nd ^ {H-1} & \mbox{if } H\in \left( \frac{3}{4}, 1\right)
\end{cases},
\end{equation*}
where $d_W$ denotes the Wasserstein distance defined in Section \ref{secdistances}.
\end{theorem}
In the case of independent entries,  the proof of our main result is based on the Stein-Malliavin calculus and the characterization of independent random variables in Wiener chaos while when the entries of the initial matrix $\X_{n,d}$ are correlated, we use the properties of random variables in the second Wiener chaos and in particular the behavior of the increments of the Rosenblatt process. 
\\~\\
The paper is organized as follows. In Section 2, we recall several
facts related to the distance between the probability distributions of
random matrices and random vectors, as well as the basics of Wiener space analysis and Malliavin calculus. In Section 3, we analyze the fluctuations of the Wishart matrix constructed from a matrix with independent entries in an arbitrary Wiener chaos,  while in Section 4 we treat the situation where the elements of the starting matrix $\X_{n,d}$ are non-Gaussian and partially correlated.

\section{Preliminaries}
In this preliminary part, we recall  some facts related to the concept
of distance between the probability distributions of
random matrices and random vectors and we introduce the tools of the Malliavin calculus needed in the sequel.

\subsection{Distances between random matrices}
\label{secdistances}

We will use the Wasserstein distance between two random matrices
taking values in $\mathcal{M}_n(\mathbb{R})$, which denotes the space
of $n \times n$ real matrices. Given two
$\mathcal{M}_n(\mathbb{R})$-valued random matrices $\mathcal{X}$ and $\mathcal{Y}$, the Wasserstein
distance between them is given by 
\begin{equation*}
d_W \left(\mathcal{X}, \mathcal{Y}  \right) =
\sup_{\norm{g}_{\on{Lip}}\leq 1} \abs{\mathbb{E}\left( g(\mathcal{X})
  \right) -\mathbb{E}\left( g(\mathcal{Y}) \right)},
\end{equation*}
where the Lipschitz norm $\norm{\cdot}_{\on{Lip}}$ of $g \colon \mathcal{M}_{n}(\mathbb{R}) \to \mathbb{R} $ is defined by  
 
\begin{equation*}
\norm{g}_{\on{Lip}} = \sup_{A \neq B\in \mathcal{M}_{n}(\mathbb{R})}\frac{ \abs{ g(A)- g(B)} }{\norm{A-B}_{\on{HS}} },
\end{equation*}
with $\norm{\cdot}_{\on{HS}}$ denoting the Hilbert-Schmidt norm on $\mathcal{M}_{n}(\mathbb{R})$.
\\~\\
With this definition at hand, we recall the definition of the notion
of $\phi$-closeness between random matrices.
\begin{definition}
  \label{phicloseness}
For every $n \geq 1$, let $\left\{ \mathcal{A}_{n,d} \colon d\geq 1
\right\}$ and $\left\{ \mathcal{B}_{n,d} \colon d\geq 1 \right\}$ be
two families of $n \times n$ random matrices. Let $\phi \colon \N
\times \N \to \R_+$ be given. Then, $\mathcal{A}_{n,d}$ is said to be
$\phi$-close to $\mathcal{B}_{n,d}$ if
$d_W(\mathcal{A}_{n,d},\mathcal{B}_{n,d})$ converges to zero as $n,d
\to \infty$ and $\phi(n,d) \to 0$.
\end{definition}
\noindent We will also make use of the Wasserstein distance between random
vectors, defined analogously as in the matrix case. Namely, if $X, Y$
are two $n$-dimensional random vectors, then the Wasserstein distance
between them is defined to be
\begin{equation}
  \label{dw}
d_W \left(X, Y  \right) =
\sup_{\norm{g}_{\on{Lip}}\leq 1} \abs{\mathbb{E}\left( g(X)
  \right) -\mathbb{E}\left( g(Y) \right)},
\end{equation}
where the Lipschitz norm $\norm{\cdot}_{\on{Lip}}$ of $g \colon \R^n \to \mathbb{R} $ is defined by  
 
\begin{equation*}
\norm{g}_{\on{Lip}} = \sup_{x \neq y \in \R^n}\frac{ \abs{ g(x)- g(y)} }{\norm{x-y}_{\R^n} },
\end{equation*}
with $\norm{\cdot}_{\R^n}$ denoting the Euclidean norm on $\R^n$.
\\~\\
If $\mathcal{X} = \left( X_{ij} \right)_{1 \leq i,j \leq n}$ is an $n
\times n$ symmetric random matrix, we associate to it its
``half-vector'' defined to be the $n(n+1)/2$-dimensional random vector
\begin{equation}\label{half}
\mathcal{X}^{\on{half}} = \left( X_{11}, X_{12} \ldots, X_{1n},
  X_{22}, X_{23}, \ldots , X_{2n}, \ldots, X_{nn}\right).
\end{equation} 
It turns out that, in the case of two symmetric matrices, the
Wasserstein distance between said matrices can be bounded from above
by a constant multiple of the Wasserstein distance between their
associated half-vectors. More specifically, we have the following
lemma (see \cite[Lemma 2.2]{nourdin_asymptotic_2018}).
\begin{lemma}\label{ll1}
Let $\X, \Y$ be two symmetric random matrices with values in $\mathcal{M}_{n}(\mathbb{R}).$ Then
\begin{equation*}
d _{W}(\X, \Y) \leq \sqrt{2} d_{W} (\X^{\on{half}}, \Y^{\on{half}}),
\end{equation*}
where $\X^{\on{half}}, \Y^{\on{half}}$ are the associated half-vectors
defined in \eqref{half}.
\end{lemma}


\subsection{Elements of Malliavin calculus}
\label{app}

We briefly describe the main tools from analysis on Wiener space that
we will need in this paper. For a complete treatment of this topic, we
refer the reader to the monographs \cite{nualart_malliavin_2006} or \cite{nourdin_normal_2012}.
\\~\\
Let $\mathfrak{H}$ be a real separable Hilbert space and $\left\{ W(h)
  \colon h \in \mathfrak{H} \right\}$ an isonormal Gaussian process
indexed by it, that is, a centered Gaussian family of random variables
such that $\E \left( W(h)W(g) \right) = \left\langle h,g
\right\rangle_{\mathfrak{H}}$. Denote by  $I_{n}$ the multiple Wiener (or Wiener-It\^o) stochastic
integral  of order $n \geq 0$ with respect to
$W$ (see \cite[Section 1.1.2]{nualart_malliavin_2006}). The mapping $I_{n}$ is actually an
isometry between the Hilbert space $\mathfrak{H}^{\odot n}$ (symmetric
tensor product) equipped with the scaled norm
$\frac{1}{\sqrt{n!}}\norm{\cdot}_{\mathfrak{H}^{\otimes n}}$ and the
Wiener chaos of order $n$, which is defined as the closed linear span
of the random variables $$\left\{ H_n(W(h)) \colon h \in \mathfrak{H},\
  \norm{h}_{\mathfrak{H}}=1 \right\},$$ where $H_{n}$ is the $n$-th Hermite
polynomial given by $H_{0}=1$ and for $n\geq 1$
\begin{equation*}
H_{n}(x)=\frac{(-1)^{n}}{n!} \exp \left( \frac{x^{2}}{2} \right)
\frac{d^{n}}{dx^{n}}\left( \exp \left( -\frac{x^{2}}{2}\right)
\right), \quad x\in \mathbb{R}.
\end{equation*}
Multiple Wiener integrals enjoy the following isometry property: for
any integers $m,n \geq 1$,
\begin{equation}
\label{iso}
\mathbb{E}\left(I_{n}(f) I_{m}(g) \right) = \mathds{1}_{\left\{ n=m \right\}} n! \langle \tilde{f},\tilde{g}\rangle _{\mathfrak{H}^{\otimes n}},
\end{equation}
where $\tilde{f} $ denotes the symmetrization of $f$ and we recall that $I_{n}(f) = I_{n} ( \tilde{f} )$.
\\~\\
Recall the multiplication formula satisfied by multiple Wiener
integrals: for any integers  $n,m \geq 1$, and any $f\in \mathfrak{H}^{\odot n}$ and
$g\in \mathfrak{H}^{\odot m}$, it holds that
\begin{equation}\label{prod}
I_n(f)I_m(g)= \sum _{r=0}^{n\wedge m} r! \binom{n}{r}\binom{m}{r} I_{m+n-2r}(f\otimes _r g),
\end{equation}
where the $r$-th contraction of $f$ and $g$ is defined by, for $0\leq r\leq m\wedge n$, 
\begin{equation}
  \label{contra}
f\otimes _r g = \sum_{i_1,\ldots , i_r =1}^{\infty} \left\langle f,
  e_{i_1}\otimes \cdots \otimes e_{i_r}
\right\rangle_{\mathfrak{H}^{\otimes r}} \otimes \left\langle g,
  e_{i_1}\otimes \cdots \otimes e_{i_r}
\right\rangle_{\mathfrak{H}^{\otimes r}},
\end{equation}
with $\left\{ e_i \colon i \geq 1 \right\}$ denoting a
complete orthonormal system in $\mathfrak{H}$. 
\\~\\
Recall that any square integrable random variable $F$ which is measurable with respect to the $\sigma$-algebra generated by $W$ can be expanded into an orthogonal sum of multiple Wiener integrals:
\begin{equation}
\label{sum1} F=\sum_{n=0}^\infty I_{n}(f_{n}),
\end{equation}
where $f_{n}\in \mathfrak{H}^{\odot n}$ are (uniquely determined)
symmetric functions and $I_{0}(f_{0})=\mathbb{E}\left(  F\right)$.
\\~\\
Let $L$ denote the Ornstein-Uhlenbeck operator, whose action on a
random variable $F$ with chaos decomposition \eqref{sum1} and such
that $\sum_{n=1}^{\infty} n^{2}n! \norm{f_n}^2_{\mathfrak{H}^{\otimes n}}<\infty$ is given by
\begin{equation*}
LF=-\sum_{n=1}^{\infty} nI_{n}(f_{n}).
\end{equation*}
For $p>1$ and $\alpha \in \mathbb{R}$ we introduce the Sobolev-Watanabe space $\mathbb{D}^{\alpha ,p }$  as the closure of
the set of polynomial random variables with respect to the norm
\begin{equation*}
\Vert F\Vert _{\alpha , p} =\Vert (I -L) ^{\frac{\alpha }{2}} F \Vert_{L^{p} (\Omega )},
\end{equation*}
where $I$ represents the identity operator. We denote by $D$  the Malliavin
derivative that acts on smooth random variables of the form $F=g(W(h_1),
\ldots , W(h_n))$, where $g$ is a smooth function with compact support
and $h_i \in \mathfrak{H}$, $1 \leq i \leq n$. Its action on such a random
variable $F$ is given by
\begin{equation*}
DF=\sum_{i=1}^{n}\frac{\partial g}{\partial x_{i}}(W(h_1), \ldots , W(h_n)) h_{i}.
\end{equation*}
The operator $D$ is closable and continuous from $\mathbb{D}^{\alpha , p} $ into $\mathbb{D} ^{\alpha -1, p} \left( \mathfrak{H}\right).$

\section{Random matrices with independent chaotic entries}
\noindent In this section, we consider random matrices $\X_{n,d} = \left( X_{ij}
\right)_{1 \leq i \leq n,\ 1 \leq j \leq d}$ with independent
entries belonging to arbitrary order Wiener chaoses associated with an
isonormal Gaussian process $W=\left\{ W(h) \colon h \in \mathfrak{H}
\right\}$ as introduced in Subsection \ref{app}. Moreover, we assume
that the elements on the same row of the matrix $\X_{n, d}$ belong to
the same Wiener chaos, while the order of the chaos may change from
one row to another. In other words, we assume that for every
$1 \leq i \leq n$ and for every $1 \leq j \leq d$,

\begin{equation}
\label{xij1}
X_{ij}= I_{q_i} (f_{ij}),
\end{equation}
with $f_{ij}\in \mathfrak{H}^{\odot q_{i}  }$, where  the integer numbers $q_{i}$ for
every $1 \leq i \leq n$  are all in the set $\{1,2,\ldots, N_{0}\}$
with $N_{0}\geq 1$ being an integer. Here and in the sequel, $I_{q}$ denotes the
multiple Wiener integral of order $q$ with respect to $W$ introduced in Subsection \ref{app}. 
\\~\\
We do not assume that the entries have the same probability distribution, only that they have the same second and fourth moments, i.e., for every $1 \leq i \leq n$ and $1 \leq j \leq d$,

\begin{equation}
  \label{m2m4}
  \E \left(  X^ {2}_{ij}  \right)= q_{i}!
  \norm{f_{ij}}_{\mathfrak{H}^{\otimes q_i}}^2 =1 \quad \mbox {and} \quad
  \E \left(  X^ {4}_{ij}  \right) = m_{4}. 
\end{equation}
Consider the centered Wishart matrix (which is what will be referred to
as Wishart matrix in the sequel)  $\W_{n, d}= \left( W_{ij}\right)_{1\leq i,j\leq n}$ defined by  
\begin{equation}
\label{wishart}
 \W_{n, d} = \frac{1}{d} \X_{n,d}\X_{n,d}^{T} - \I_{n},
\end{equation}
where $\I_{n}$ denotes the identity matrix of $\mathcal{M}_n(\R)$, and $\X^ {T}$
stands for the tranpose of the matrix $\X$. Note that the Wishart matrix $\W_{n,d}$ is a symmetric $n\times n$ matrix and its  entries  can be explicited as 
\begin{equation}
\label{wii}
W_{ii}= \frac{1}{d} \sum_{k=1}^ {d} \left( X_{ik}^ {2}-1 \right), \quad i=1,\ldots,n
\end{equation}
and
\begin{equation}
\label{wij}
W_{ij} = \frac{1}{d} \sum_{k=1}^ {d} X_{ik}X_{jk}, \quad 1\leq i \neq j\leq n.
\end{equation}
Note that the independence of the entries $X_{ij}$ and assumption \eqref{m2m4} yield, for any $1\leq i\leq n$,
\begin{equation*}
\E \left( W_{ii}^ {2} \right)  = \frac{1}{d^ {2} } \sum_{k=1}^ {d} E
\left(   \left( X_{ik}^ {2}-1\right) ^ {2} \right) = \frac{m_4-1}{d},
\end{equation*}
and for $1\leq i \neq j\leq n$,
\begin{equation*}
\E \left( W_{ij} ^ {2} \right) = \frac{1}{d^ {2}} \sum_{k=1} ^ {d} \E \left(   X_{ik} ^ {2}\right) \E \left(   X_{jk} ^ {2}\right)= \frac{1}{d}.
\end{equation*}
Based on this observation, we define the renormalized Wishart matrix $ \widetilde{\W} _{n, d} = (\widetilde{W}_{ij}) _{1\leq i,j\leq n} $ as 
\begin{equation}
  \label{hatwij}
  \widetilde{W}_{ij}= \sqrt{d} W_{ij}
\end{equation}
for all $1\leq i, j\leq n$.
\\~\\
Also, consider the GOE matrix $\Z_n = \left( Z_{ij} \right)_{1 \leq
  i,j \leq n}$ with entries given by
\begin{equation}
\label{goe}
\begin{cases} Z_{ii} \sim N(0, m_4-1) & \mbox{for } 1 \leq i \leq n\\
Z_{ij} \sim N(0,1) & \mbox{for } 1 \leq i < j \leq n\\
Z_{ij} = Z_{ji} &\mbox{for } 1 \leq j < i \leq n
\end{cases},
\end{equation}
where the entries $\left\{ Z_{ij} \colon i \leq j \right\}$ are independent.
\begin{remark}
Note that proving Theorem \ref{mainresult-independence} entails proving that the matrices
$\widetilde{\W} _{n, d}$ and $\Z_n$ are $\phi$-close for $\phi(n,d) =
\frac{n^3}{d}$ (as introduced in Definition \ref{phicloseness}).
\end{remark}
\noindent As pointed out in Subsection \ref{secdistances}, assessing the Wasserstein
distance between symmetric random matrices can be shifted to the
problem of estimating the Wasserstein distance between associated random
vectors (see Lemma \ref{ll1}). In our context, a helpful result in
this direction is \cite[Theorem 6.1.1]{nourdin_normal_2012}, which we restate here
for convenience. 
\begin{theorem}[Theorem 6.1.1 in \cite{nourdin_normal_2012}]
\label{tt1} 
Fix $m \geq 2$, and let  $F= (F_{1},\ldots, F_{m}) $ be a centered
$m$-dimensional random vector with $F_{i}\in \mathbb{D}^{1, 4}$ for
every $i=1,\ldots,m$. Let $C \in \mathcal{M}_m(\R)$ be a symmetric and
positive definite matrix, and let $Z\sim N_m(0, C)$. Then,
\begin{equation*}
d_{W}(F, Z) \leq  \norm{C^{-1}}_{\on{op}}
\norm{C}_{\on{op}}^{1/2} \sqrt{ \sum_{i, j=1}^{m} \E \left(\left(
      C_{ij} - \left\langle DF_{i}, -DL^{-1} F_{j} \right\rangle_{\mathfrak{H}}  \right) ^{2}\right)},
\end{equation*}
where $\norm{\cdot}_{\on{op}}$ denotes the operator norm on $\mathcal{M}_m(\R)$.
\end{theorem}


\subsection{Independent random variables in Wiener chaos}

This section prepares the proof of Theorem
\ref{mainresult-independence} by providing results related to the
independence of multiple Wiener integrals. By a standard argument
based on the fact that separable Hilbert spaces are isometrically isomorphic, we may assume, when
it serves the clarity of our exposition, that $\mathfrak{H}=L^ {2}(T, \mathcal{B}, \mu)$ where
$\mu$ is a $\sigma$-finite measure without atoms.
\\~\\
Recall that the entries of the matrix $\X$, on which our Wishart
matrices are based, are independent multiple Wiener integrals of possibly
different orders. The independence of random variables in Wiener chaos
can be characterized in terms of their kernels via the celebrated
\"Ust\"unel-Zakai criterion (see \cite{ustunel_independence_1989}), which we will
intensively make use of in the sequel. We recall the criterion here for
convenience.
\begin{theorem}[\"Ust\"unel-Zakai \cite{ustunel_independence_1989}]
For any $n,m \geq 1$, let $f \in \mathfrak{H}^{\otimes n}$ and $g \in
\mathfrak{H}^{\otimes m}$. The multiple Wiener integrals $I_{n}(f)$ and $I_{m}(g)$ are independent if and only if 
\begin{equation}
  \label{f1}
  f\otimes _{1} g=0 \mbox{ almost everywhere on }  \mathfrak{H} ^{\otimes m+n-2}.
\end{equation}
\end{theorem}
\begin{remark}
Relation \eqref{f1}  also implies that
$$f\otimes _{r} g=0 \mbox{ almost everywhere on }  \mathfrak{H} ^{\otimes m+n-2r}$$
for all $1 \leq r\leq n \wedge m$. 
\end{remark}
\noindent We will also need the notion of strong independence of random
variables introduced in \cite{bourguin_cramer_2011} (to which we refer for various properties of strongly independent random variables).
\begin{definition}
  Two random variables $X$ and $Y$ with Wiener chaos decomposition
\begin{equation*}
X=\sum_{n=0}^{\infty}I_{n}(f_{n})\quad \mbox{and}\quad Y=\sum_{m=0}^{\infty} I_{m} (g_{m}),
\end{equation*}
where $f_{n}\in \mathfrak{H}^ {\odot n}$, $g_{m}\in \mathfrak{H}^{\odot m}$ for
every $n,m\geq 0$, are said to be strongly independent if every chaos component of $X$ is independent of every chaos component of $Y$, i.e., for every $n,m\geq 0$, the random variables $I_{n} (f_{n}) $ and $I_{m}(g_{m}) $ are independent.
\end{definition}
\noindent The following lemma assesses the strong independence of squares of
chaotic random variables.
\begin{lemma}\label{ll2}
Let $X=I_{n}(f)$, $f\in \mathfrak{H}^{\odot n}$ and $Y=I_{m}(g)$, $g\in
\mathfrak{H}^ {\odot m}$ be independent. Then, the random variables $ X ^ {2}$ and $Y^ {2}$ are strongly independent. 
\end{lemma}
\begin{proof}
By the product formula for multiple Wiener integrals \eqref{prod},
$$X ^ {2}= \sum_{r_{1}=0} ^ {n} r_{1}! \binom{n}{r_1} ^ {2} I_{2n-2r_{1}} (f\otimes _{r_{1}} f) $$
and
$$ Y^ {2}= \sum_{r_{2}=0}^ {m}   r_{2}! \binom{m}{r_2} ^ {2} I_{2m-2r_{2}} (g\otimes _{r_{2}} g).$$
It suffices to show that for every $0\leq r_{1}\leq n-1$ and $0\leq
r_{2} \leq m-1$, the random variables $I_{2n-2r_{1}} (f\otimes
_{r_{1}} f)$ and $I_{2m-2r_{2}} (g\otimes _{r_{2}} g)$ are independent, which by \eqref{f1} is equivalent to
\begin{equation}\label{27i-3}
\left(f\tilde{\otimes }_{r_{1}}f\right)   \otimes _{1} \left(
  g\tilde{\otimes }_{r_{2}}g\right) =0
\end{equation}
almost everywhere on $\mathfrak{H} ^ {\otimes 2n + 2m -2r_1 -2r_2}$. By
the definition of contractions \eqref{contra}, with $\mathfrak{S}_{n}$
denoting the group of permutations of $\{1,\ldots, n\}$, we have
\begin{eqnarray*}
&& (f\tilde{\otimes }_{r_{1}} f) (t_{1},\ldots, t_{2n-2r_{1}}) \\
  &&\qquad =\frac{1}{(2n-2r_{1})!}\sum_{\sigma \in \mathfrak{S}_{2n-2r_{1}}} \int_{T ^{r_{1}}} f(u_{1},\ldots,u_{r_{1}},t_{\sigma(1)},\ldots,t_{\sigma(n-r_{1})})\\
  && \qquad\qquad\qquad\qquad\qquad\quad f(u_{1},\ldots,
     u_{r_{1}}, t_{\sigma (n-r_{1}+1)}, \ldots, t_{\sigma
     (2n-2r_{1})}) du_{1}\cdots du_{r_{1}}.
\end{eqnarray*}
Similarly,
\begin{eqnarray*}
&& (g\tilde{\otimes }_{r_{2}} g) (t_{1}, \ldots, t_{2m-2r_{2}}) \\
  &&\qquad =\frac{1}{(2m-2r_{2})!}\sum_{\tau \in \mathfrak{S}_{2m-2r_{2}}} \int_{T ^{r_{2}}} g(u_{1},\ldots,u_{r_{2}},t_{\tau(1)},\ldots,t_{\tau(m-r_{2})})\\
  && \qquad\qquad\qquad\qquad\qquad\quad g(u_{1},\ldots,
     u_{r_{2}}, t_{\tau (m-r_{2}+1)},\ldots, t_{\tau (2m-2r_{2})})
     du_{1}\cdots du_{r_{2}}.
\end{eqnarray*}
Hence, we can write
\begin{eqnarray}
  \label{27i-1}
&&\left( \left(f\tilde{\otimes }_{r_{1}}f\right)   \otimes _{1} \left( g\tilde{\otimes }_{r_{2}}g\right)\right) (t_{1},\ldots, t_{2n-2r_{1}+ 2m-2r_{2}-2}) \nonumber\\
&&\qquad = \int_{T} (f\tilde{\otimes}_{r_{1}} f) (t_{1},\ldots,
   t_{2n-2r_{1}-1}, x) \nonumber \\
  &&\qquad\qquad\qquad\qquad\qquad (g\tilde{\otimes }_{r_{2}}g)(t_{2n-2r_{1}},\ldots, t_{2n-2r_{1}+2m-2r_{2}-2}, x)dx.
\end{eqnarray}
Note that for a symmetric function $h\in \mathfrak{H} ^ {\odot n}$, it
holds that
\begin{equation*}
\tilde{h}(t_{1},\ldots, t_{n-1}, x) = \frac{1}{n!} \sum_{\sigma \in \mathfrak{S}_{n-1} } \sum_{i=1}^ {n} h\left(t_{\sigma (1)},\ldots, t_{\sigma (i-1)}, x, t_{\sigma(i+1)},\ldots, t_{\sigma (n-1) }\right),
\end{equation*}
so that by plugging the above identity into \eqref{27i-1}, we get
\begin{eqnarray*}
&&\left[ \left(f\tilde{\otimes }_{r_{1}}f\right)   \otimes _{1} \left(
   g\tilde{\otimes }_{r_{2}}g\right)\right] (t_{1},\ldots,
   t_{2n-2r_{1}+ 2m-2r_{2}-2}) \nonumber\\
&&\qquad\qquad = \frac{1}{(2n-2r_{1}-1)!(2m-2r_{2}-1)!} \sum_{\sigma \in
    \mathfrak{S}_{2n-2r_{1}-1}, \tau \in \mathfrak{S}_{2m-2r_{2}-1}}\sum_{i=1}^ {2n-2r_{1}} \sum_{j=1} ^ {2m-2r_{2}} \\
&&\int_ {T} (f\otimes_{r_{1}} f) (t_{\sigma (1)},\ldots, t_{\sigma (i-1)}, x, t_{\sigma (i+1)},\ldots, t_{\sigma (2n-2r_{1}-1)})\\
&& \qquad\qquad\qquad\qquad\qquad (g\otimes _{r_{2}}g) (t_{\tau(1)},\ldots,t_{\tau(j-1)}, x, t_{\tau(j+1)},\ldots, t_{\tau (2m-2r_{2}-1)})dx.
\end{eqnarray*}
To obtain \eqref{27i-3}, it suffices to show that for all $1\leq i\leq
2n-2r_{1}$ and $1\leq j\leq 2m-2r_{2}$,
\begin{eqnarray}
\label{27i-4}
&&\int_ {T} (f\otimes_{r_{1}} f) (t_{\sigma (1)},\ldots, t_{\sigma (i-1)}, x, t_{\sigma (i+1)},\ldots, t_{\sigma (2n-2r_{1}-1)})\\
&& \qquad\qquad\qquad\qquad\qquad (g\otimes _{r_{2}}g)
   (t_{\tau(1)},\ldots,t_{\tau(j-1)}, x, t_{\tau(j+1)},\ldots, t_{\tau
   (2m-2r_{2}-1)})dx=0\nonumber
\end{eqnarray}
almost everywhere with respect to $t_{1},\ldots, t_{2n+2m-2r_{1}-2r_{2}-2}$. 
\\~\\
Assume that $1\leq i\leq n-r_{1}$ and $1\leq j\leq m-r_{2}$ (the other
cases can be dealt with in the same way). Then, we have
\begin{eqnarray*}
&&\int_ {T} (f\otimes_{r_{1}} f) (t_{\sigma (1)},\ldots, t_{\sigma (i-1)}, x, t_{\sigma (i+1)},\ldots, t_{\sigma (2n-2r_{1}-1)})\\
&&\qquad\qquad\qquad\qquad\qquad  (g\otimes _{r_{2}}g) (t_{\tau(1)},\ldots,t_{\tau(j-1)}, x, t_{\tau(j+1)},\ldots, t_{\tau (2m-2r_{2}-1)})dx\\
&&\quad =\int_{T} \int_{T^ {r_{1}}}du_{1}\cdots du_{r_{1}}
   f(u_{1},\ldots, u_{r_{1}}, x, t_{\sigma (1)},\ldots, t_{\sigma
   (n-r_{1}-1)} ) \nonumber \\
  && \qquad\qquad\qquad\qquad\qquad\qquad\qquad\qquad\qquad\qquad f(t_{\sigma (n-r_{1})},\ldots, f(t_{\sigma (2n-2r_{1}-1)})\\
&&\qquad \times \int _{T^ {r_{2}}} dv_{1}\cdots dv_{r_{2}}
   g(v_{1},\ldots,v_{r_{2}}, x,
   t_{\tau(1)},\ldots,t_{\tau(m-r_{2}-1)}) \nonumber \\
  &&\qquad\qquad\qquad\qquad\qquad\qquad\qquad\qquad\qquad\qquad g(t_{\tau (m-r_{2})},\ldots, t_{\tau (2m-2r_{2}-1)})dx. 
\end{eqnarray*} 
Now, for almost every $u_{1},\ldots, u_{r_{1}},
v_{1},\ldots,v_{r_{2}},  t_{\sigma (1)},\ldots, t_{\sigma
  (n-r_{1}-1)}, t_{\tau(1)},\ldots,t_{\tau(m-r_{2}-1)}$, \eqref{f1}
implies that

$$\int_{T}  f(u_{1},\ldots, u_{r_{1}}, x, t_{\sigma (1)},\ldots, t_{\sigma (n-r_{1}-1)} ) g(v_{1},\ldots,v_{r_{2}}, x, t_{\tau(1)},\ldots,t_{\tau(m-r_{2}-1)})dx=0,$$
which implies \eqref{27i-4} and in turn \eqref{27i-3}. 
\end{proof}
~\\
\noindent The following lemma is the statement of \cite[Lemma 2]{bourguin_cramer_2011}.

\begin{lemma}\label{ll3}
Let $X,Y$ be centered, strongly  independent random variables in $\mathbb{D}^ {1,2}$. Then
$$\langle DX, -DL ^ {-1}Y\rangle _{\mathfrak{H}}= \langle DY, -DL^ {-1}X \rangle _{\mathfrak{H}}=0.$$
\end{lemma}
\noindent We prove another consequence of strong independence needed
later in the paper.

\begin{lemma}\label{ll44}
Let $X, Y$ be strongly  independent random variables in $\mathbb{D}^ {1,2}$.
\begin{enumerate}[(i)]
\item The random variables  $\langle DX, -DL ^ {-1}X\rangle _{\mathfrak{H}} $ and $\langle DY, -DL ^ {-1}Y\rangle _{\mathfrak{H}}$ are strongly independent.
\item The random variables  $X $ and $\langle DY, -DL^ {-1}Y\rangle _{\mathfrak{H}}$ are strongly independent.
\end{enumerate}
\end{lemma}
\begin{proof}
Let us prove $(i)$ (the proof of $(ii)$ follows in a similar way by
the same arguments, and an analogous result has been proved in
\cite[Lemma 1]{bourguin_cramer_2011}).   Assume 
\begin{equation*}
X=\sum_{n=0}^{\infty} I_{n}(f_{n}) \quad \mbox{and}\quad Y=\sum_{m=0}^{\infty} I_{m}(g_{m}),
\end{equation*}
where $f_{n} \in \mathfrak{H} ^ {\odot n}$ and $g_{m}\in \mathfrak{H}^ {\odot m}$ for every
$n,m\geq 0$.  Then, we have
\begin{equation*}
D_{\theta}X= \sum_{n\geq 1} n I_{n-1}(f_{n} (\cdot , \theta)) \quad \mbox{and}\quad  -D_{\theta}L ^ {-1} X= \sum_{n\geq 1}  I_{n-1}(f_{n} (\cdot , \theta)),
\end{equation*}
where $I_{n-1}(f_{n} (\cdot , \theta))$ denotes the multiple Wiener
integral of the function $$(t_{1},\ldots, t_{n-1}) \mapsto
f_{n}(t_{1},\ldots, t_{n-1}, \theta).$$
Then, it holds that
\begin{eqnarray*}
\langle DX, -DL^ {-1} X \rangle _{\mathfrak{H}}&=& \sum _{n_{1}, n_{2}= 1}^{\infty} n_{1} \int_{T} I_{n_{1}-1} \left(f_{n_{1}}(\cdot, x)  \right)I_{n_{2}-1}\left( f_{n_{2}} (\cdot, x) \right)dx \\
&=&\sum _{n_{1}, n_{2}= 1}^{\infty} n_{1}\sum_{r=0} ^ {n_{1} \wedge n_{2}-1} \binom{n_1}{r}\binom{n_2}{r} I _{n_{1}+n_{2}-2r-2} (f_{n_{1}}\tilde{ \otimes } _{r+1} f_{n_{2}}). 
\end{eqnarray*}
Similarly, 
\begin{equation*}
\langle DY, -DL^ {-1} Y \rangle _{\mathfrak{H}} = \sum _{m_{1}, m_{2}= 1}^{\infty} m_{1}\sum_{r=0} ^ {m_{1} \wedge m_{2}-1} \binom{m_1}{r}\binom{m_2}{r} I _{m_{1}+m_{2}-2r-2} (g_{m_{1}}\tilde{ \otimes } _{r+1} g_{m_{2}}). 
\end{equation*}
The conclusion is obtained if we prove that for every $0\leq r_{1}\leq
n_{1}\wedge n_{2}-1$ and for every $0\leq r_{2} \leq m_{1} \wedge
m_{2}-1$, the random variables $ I _{n_{1}+n_{2}-2r-2}
(f_{n_{1}}\tilde{ \otimes } _{r_{1}+1} f_{n_{2}})$ and $  I
_{m_{1}+m_{2}-2r-2} (g_{m_{1}}\tilde{ \otimes } _{r_{2}+1} g_{m_{2}})$
are independent, or equivalently, that 
\begin{equation}\label{27i-5}
(f_{n_{1}}\tilde{ \otimes } _{r_{1}+1} f_{n_{2}})\otimes _{1}  (g_{m_{1}}\tilde{ \otimes } _{r_{2}+1} g_{m_{2}})=0 \mbox{ a.e. } 
\end{equation}
Since for every $n,m\geq 0$, we have $f_{n}\otimes _{1} g_{m}=0$ almost everywhere on $T^ {m+n-2}$, \eqref{27i-5} follows from the proof of Lemma \ref{ll2}. 
\end{proof}
~\\
\noindent Let us illustrate what the above results on strong
independence imply about the entries of the matrix $\X_{n,d}$. We
begin by introducing some notation. For every $1\leq i\leq n$ and for every $1\leq k,l\leq d$, we define
\begin{equation}
\label{fikl}
F_{ikl} = \langle D(X_{ik}^ {2}-1),  -DL^ {-1} (X_{il} ^
{2}-1)\rangle _{\mathfrak{H}}.
\end{equation}
We then have the following lemma.
\begin{lemma}\label{ll4}
Let the above notation prevail. 
\begin{enumerate}[(i)]
\item If $k\neq l$, $F_{ikl} =0$ almost surely.
\item For every $k,l=1,\ldots, d$ with $k\neq l$, the random variables $F_{ikk}$ and $F_{ill}$ are independent. 
\end{enumerate}
\end{lemma}
\begin{proof}
By Lemma \ref{ll2}, $X_{ik}^ {2}$ and   $ X_{il}^ {2}$  are strongly
independent random variables. Lemma \ref{ll3} yields $(i)$, and Lemma 4 implies $(ii)$.  
\end{proof}

\subsection{Proof of Theorem \ref{mainresult-independence}}

This subsection is dedicated to the proof of Theorem
\ref{mainresult-independence}. We restate it here for convenience.
\begin{customthm}{1}
Consider the renormalized  Wishart matrix $\widetilde{\W}_{n, d}$ with
entries given by \eqref{hatwij}. Then for every $n\geq 1$,
$\widetilde{\W}_{n, d}$  converges in distribution componentwise, as $d\to \infty$,  to
the matrix $\Z_n$ given by \eqref{goe}. Moreover, there exists a positive constant $C$ such that for every $n,d\geq
1$, 
\begin{equation*}
d_{W} (\widetilde{\W}_{n, d}, \Z_n ) \leq C\sqrt{ \frac{n ^ {3}}{d}}.
\end{equation*}
\end{customthm}
\begin{proof}
Lemma \ref{ll1} combined with Theorem \ref{tt1} implies that we need
to estimate the quantity
\begin{equation*}
\E \left(\left(  \langle D\widetilde{W} _{ij}, -DL^ {-1}\widetilde{W}_{ab} \rangle_{\mathfrak{H}} -\E\left( Z_{ij}Z_{ab}\right) \right)^2\right)
\end{equation*}
for every $1\leq i,j, a, b\leq n$ with $i\leq j$ and $a\leq b$, and
$Z_{ij}$ as in \eqref{goe}. Note that $\E \left( Z_{ii}^ {2}\right)=
m_{4}-1$, $\E\left(Z_{ij}^ {2}\right)=1$ if $i\neq j$, and $\E \left(
  Z_{ij}Z_{ab}\right)=0$ if $(i,j)\neq (a,b)$.
\\~\\
\noindent  {\bf Step 1: calculation of $ \E\left( \left(  \langle
      D\widetilde{W}_{ii}, -DL^{-1} \widetilde{W}_{ii}\rangle_{\mathfrak{H}} -(m_{4}
      -1) \right) ^{2}\right)  $}.
\\~\\
By \eqref{wii} and the strong independence proved in Lemma \ref{ll4}, for every $1\leq i\leq n$, it
holds that
\begin{eqnarray*}
 \langle D\widetilde{W}_{ii}, -DL^{-1}\widetilde{ W}_{ii}\rangle_{\mathfrak{H}} &=& \frac{1}{d} \sum_{k,l=1}^ {d}  \langle D(X_{ik}^{2}-1), -DL^{-1} ( X_{il}^{2}-1) \rangle_{\mathfrak{H}}\\
&=&\frac{1}{d} \sum_{k=1} ^{d} \langle D(X_{ik}^{2}-1), -DL^{-1} ( X_{ik}^{2}-1) \rangle_{\mathfrak{H}} = \frac{1}{d} \sum_{k=1} ^{d}  F_{ikk},
\end{eqnarray*}
where $F_{ikk}$ is given by \eqref{fikl}.  Since  for every $G\in
\mathbb{D}^ {1,2}$, $\E \left( G^{2} \right)  = \E \left(\langle DG,
  -DL^{-1} G\rangle_{\mathfrak{H}}  \right) $, we can write, using \eqref{m2m4},
\begin{eqnarray*}
\E  \left(\langle D\widetilde{W}_{ii}, -DL^{-1}
  \widetilde{W}_{ii}\rangle_{\mathfrak{H}}\right)&=&  \frac{1}{d}
                                                 \sum_{k=1} ^{d}
                                                 \E\left(\langle
                                                 D(X_{ik}^{2}-1),
                                                 -DL^{-1} (
                                                 X_{ik}^{2}-1)
                                                 \rangle_{\mathfrak{H}}\right)
  \\
&=&\frac{1}{d}\sum_{k=1}^ {d} \E\left(\left(X_{ik}^ {2}-1\right)^{2}\right)= m_{4}-1.
\end{eqnarray*}
Hence, we can write
\begin{eqnarray}
&&\E \left( \left(  \langle D\widetilde{W}_{ii}, -DL^{-1} \widetilde{W}_{ii}\rangle_{\mathfrak{H}}-(m_{4}-1)\right) ^{2} \right) \nonumber  \\
  && \qquad\qquad\qquad\qquad = \E \left( \left(  \langle
     D\widetilde{W}_{ii}, -DL^{-1}
     \widetilde{W}_{ii}\rangle_{\mathfrak{H}}-\E \left( \langle D\widetilde{W}_{ii}, -DL^{-1} \widetilde{W}_{ii}\rangle_{\mathfrak{H}}\right) \right)^{2}\right)\nonumber \\
&& \qquad\qquad\qquad\qquad =\frac{1}{d ^{2} } \E \left( \left(
   \sum_{k=1}^{d} \left(F_{ikk}-\E \left(F_{ikk}\right) \right)\right)
   ^{2}\right) \nonumber \\
  &&\qquad\qquad\qquad\qquad =\frac{1}{d ^{2} } \sum_{k=1}^{d} \E\left( \left(F_{ikk}-\E\left(F_{ikk}\right)\right) ^{2}\right). \label{28i-3}
\end{eqnarray}
We claim that for every $1\leq i\leq n$ and $ 1\leq k\leq d$, 
\begin{equation*}
\E\left( \left(F_{ikk}-\E\left(F_{ikk}\right)\right) ^{2}\right) \leq C(i),
\end{equation*}
where $C(i)>0$ is a constant depending on $i$, but not on $k$. In
order to prove this, we will
make use of the Wiener chaos decomposition of $F_{ikk}$, together with
\eqref{f1} and assumption \eqref{m2m4}. From \eqref{xij1} and the
product formula \eqref{prod}, for every $1\leq i\leq n$ and $1\leq k\leq
d$, it holds that
\begin{equation*}
X_{ik}^ {2} =\sum_{r=0} ^ {q_{i}} r! \binom{q_i}{r}^ {2} I_{2q_{i}-2r}(f_{ik} \otimes _{r} f_{ik} ),
\end{equation*}
\begin{equation*}
D_{\theta}(X_{ik}^{2}-1)= \sum_{r=0} ^{q_{i}-1} r! \binom{q_i}{r}^{2} (2q_{i}- 2r) I _{2q_{i} -2r-1} (( f_{ik}\otimes _{r} f_{ik}) (\cdot, \theta) ),
\end{equation*}
and
\begin{equation*}
-D_{\theta}L^{-1}(X_{ik}^{2}-1)= \sum_{r=0} ^{q_{i}-1} r! \binom{q_i}{r}^{2}I _{2q_{i} -2r-1} ( (f_{ik}\otimes _{r} f_{ik}) (\cdot, \theta) ).
\end{equation*}
This yields, for every $1\leq i\leq n$ and $1\leq k\leq d$,
\begin{eqnarray*}
F_{ikk}&=& \langle D(X_{ik}^{2}-1), -DL^{-1}(X_{ik}^{2}-1)\rangle_{\mathfrak{H}} \\
&=& \sum_{r_{1}, r_2=0} ^{q_{i}-1} r_{1}! r_{2}!
    \binom{q_i}{r_1}^{2}\binom{q_i}{r_2}^{2}(2q_{i}-2r_{1}) \nonumber
\\
  && \qquad\qquad\qquad\qquad\qquad\qquad\langle  I _{2q_{i} -2r_{1}-1} ( f_{ik}\otimes _{r_{1}} f_{ik}  ),  I _{2q_{i} -2r_{2}-1} ( f_{ik}\otimes _{r_{2}} f_{ik} )\rangle_{\mathfrak{H}}  \\
&=&  \sum_{r_{1}, r_2=0} ^{q_{i}-1} r_{1}! r_{2}!
    \binom{q_i}{r_1}^{2}\binom{q_i}{r_2}^{2}(2q_{i}-2r_{1})\\
    && \qquad\qquad \sum_{p=0} ^{(2q_{i}-2r_{1})\wedge (2q_{i}-2r_{2})-1}
    p!\binom{2q_i-2r_1 -1}{p}\binom{2q_i-2r_2-1}{p} \\
&& \qquad\qquad\qquad\qquad I_{4q_{i}-2r_{1}-2r_{2} -2(p+1) } \left(  ( f_{ik}\otimes _{r_{1}} f_{ik}  )\otimes _{p+1} ( f_{ik}\otimes _{r_{2}} f_{ik}  )\right),
\end{eqnarray*}
and hence
\begin{eqnarray}
&&F_{ikk}- \E \left(F_{ikk}\right)\nonumber \\
&&\qquad = \sum_{r_{1},r_2=0} ^{q_{i}-1} \mathds{1}_{\left\{ r_{1}\neq r_{2}
    \right\}}r_{1}! r_{2}!
    \binom{q_i}{r_1}^{2}\binom{q_i}{r_2}^{2}(2q_{i}-2r_{1}) \nonumber \\
  && \qquad\qquad\qquad\qquad\qquad\qquad\qquad \langle  I _{2q_{i} -2r_{1}-1} ( f_{ik}\otimes _{r_{1}} f_{ik} ),  I _{2q_{i} -2r_{2}-1} ( f_{ik}\otimes _{r_{2}} f_{ik}  )\rangle_{\mathfrak{H}}  \nonumber\\
&&\qquad =  \sum_{r_{1}=0} ^{q_{i}-1} \sum_{r_{2}=0} ^{q_{i}-1} r_{1}! r_{2}! \binom{q_i}{r_1}^{2}\binom{q_i}{r_2}^{2}(2q_{i}-2r_{1})\nonumber  \\
&& \qquad\qquad \sum_{p=0} ^{(2q_{i}-2r_{1})\wedge (2q_{i}-2r_{2})-1} p!
   \binom{2q_i-2r_1 -1}{p}\binom{2q_i-2r_2-1}{p} \nonumber \\
  && \qquad\qquad\qquad\qquad\qquad\qquad I_{4q_{i}-2r_{1}-2r_{2} -2(p+1) } \left(  ( f_{ik}\otimes _{r_{1}} f_{ik}  )\otimes _{p+1} ( f_{ik}\otimes _{r_{2}} f_{ik} )\right)\nonumber\\
&&\qquad + \sum_{r=0} ^{q_{i}-1} r! ^{2} \binom{q_i}{r}^{4}
   (2q_{i}-2r)  \sum_{p=0} ^{2q_{i}-2r-2}p! \binom{2q_i-2r-1}{p} ^{2}
   \nonumber \\
  && \qquad\qquad\qquad\qquad\qquad\qquad I_{4q_{i}-4r -2(p+1) } \left(  ( f_{ik}\otimes _{r} f_{ik}  )\otimes _{p+1} ( f_{ik}\otimes _{r} f_{ik} )\right).\label{28i-11}
\end{eqnarray}
Now, using the isometry property \eqref{iso} of multiple Wiener
integrals together with the bounds
$\Vert \tilde{f} \Vert_{\mathfrak{H}^{\otimes n}} \leq
\norm{f}_{\mathfrak{H}^{\otimes n}} $ and $\Vert f\otimes _{r}
g\Vert_{\mathfrak{H}^{\otimes (2n-2r)}} \leq \Vert
f\Vert_{\mathfrak{H}^{\otimes n}} \Vert g\Vert_{\mathfrak{H}^{\otimes n}} $
for every $f, g\in \mathfrak{H} ^ {\otimes n}$ and $0\leq r\leq n$, we can write
\begin{eqnarray}
&&\E \left( I_{4q_{i}-2r_{1}-2r_{2} -2(p+1) } \left(  \left( f_{ik}\otimes _{r_{1}} f_{ik}  \right)\otimes _{p+1} \left( f_{ik}\otimes _{r_{2}} f_{ik}  \right)\right)^2\right) \nonumber\\
&& \qquad\qquad\qquad = c(q_{i}, r_{1}, r_{2}, p)
   \Vert ( f_{ik}\tilde{\otimes} _{r_{1}} f_{ik} 
   )\tilde{\otimes} _{p+1} ( f_{ik}\tilde{\otimes} _{r_{2}} f_{ik} )\Vert_{\mathfrak{H}^{\otimes (4q_{i}-2r_{1}-2r_{2} -2(p+1))}} ^{2} \nonumber \\
  && \qquad\qquad\qquad \leq c(q_{i}, r_{1}, r_{2},
     p)  \Vert   f_{ik}\tilde{\otimes} _{r_{1}}
     f_{ik}\Vert_{\mathfrak{H}^{\otimes (2q_i-2r_1)}} ^{2}\Vert   f_{ik}\tilde{\otimes} _{r_{2}} f_{ik}\Vert_{\mathfrak{H}^{\otimes (2q_i-2r_2)}} ^{2}\nonumber \\
&&\qquad\qquad\qquad \leq  c(q_{i}, r_{1}, r_{2}, p)\Vert f_{ik} \Vert_{\mathfrak{H}^{\otimes q_i}} ^{8} \nonumber
\\
  && \qquad\qquad\qquad \leq  c(q_{i}, r_{1}, r_{2}, p),\label{28i-2}
\end{eqnarray}
where $c(q_{i}, r_{1}, r_{2}, p)$ is a strictly positive constant
depending on $q_{i}, r_{1}, r_{2}, p$ but not on $k$.  Now, in
\eqref{28i-11}, we use the isometry property \eqref{iso} together with \eqref{28i-2} to obtain, for every $1\leq i\leq n$ and for every $1\leq k\leq d$,
\begin{equation*}
\E \left( \left( F_{ikk}- \E\left( F_{ikk}\right)\right) ^ {2}\right)\leq C(i),
\end{equation*}
where $C(i)>0$ is a constant (depending only on $q_{i}$). Therefore,
using the above inequality and \eqref{28i-3} yields 
\begin{equation}\label{bb1}
\E\left(\left(  \langle D\widetilde{W}_{ii}, -DL^{-1}
    \widetilde{W}_{ii}\rangle_{\mathfrak{H}}-(m_{4}-1)\right) ^{2}\right)
= \frac{1}{d ^{2} } \sum_{k=1}^{d} \E\left( (F_{ikk}-\E\left(F_{ikk}\right)) ^{2}\right) \leq \frac{C(i)}{d}. 
\end{equation}
~\\
\noindent  {\bf Step 2: calculation of $ \E\left( \left(  \langle D\widetilde{W}_{ij}, -DL^{-1} \widetilde{W}_{ij}\rangle_{\mathfrak{H}} -1 \right) ^{2}\right)$ with $i<j$}.
\\~\\
Assume $1\leq i < j\leq n$. In this case, by  \eqref{xij1}, the
product formula \eqref{prod} as well as \eqref{f1}, we have  for every $1\leq k\leq d$, 
$$X_{ik}X_{jk}= I_{q_{i}+ q_{j}} (f_{ik}\otimes f_{jk}),$$
so that $X_{ik}X_{jk}$ is an element of the $(q_{i}+ q_{j})$-th Wiener chaos. Consequently, $$-DL^ {-1} (X_{ik}X_{jk})=\frac{1}{q_{i}+q_{j}} D(X_{ik}X_{jk})$$ and 
\begin{eqnarray*}
\langle D\widetilde{W}_{ij}, -DL^{-1} \widetilde{W}_{ij}\rangle_{\mathfrak{H}}&=&\frac{1}{d \left(q_{i}+ q_{j}\right)} \sum_{k,l=1} ^{d} \langle D(X_{ik}X_{jl}), -DL^ {-1} (X_{ik}X_{jk})\rangle_{\mathfrak{H}} \\
&=&  \frac{1}{d \left(q_{i}+ q_{j}\right)}\sum_{k=1} ^{d} \Vert D(X_{ik}X_{jk})\Vert _{\mathfrak{H}} ^{2} \\
&=&\frac{1}{d \left(q_{i}+ q_{j}\right)}\sum_{k=1}^{d} \left( X_{ik}^{2} \Vert DX_{jk}\Vert_{\mathfrak{H}} ^{2} + X_{jk}^{2} \Vert DX_{ik}\Vert_{\mathfrak{H}} ^{2}\right).
\end{eqnarray*}
On the other hand, since $\E\left(\Vert DX_{ik}\Vert_{\mathfrak{H}} ^
  {2}\right)=q_{i}$ and $ \E\left( \Vert DX_{jk} \Vert_{\mathfrak{H}} ^
  {2}\right)=q_{j}$, we have 
\begin{eqnarray*}
&& \E\left( \langle D\widetilde{W}_{ij}, -DL^{-1}
   \widetilde{W}_{ij}\rangle_{\mathfrak{H}}\right) \\
  && \qquad\qquad = \frac{1}{d \left(q_{i}+ q_{j}\right)}\sum_{k=1}^{d} \left( \E\left(X_{ik}^{2}\right) \E\left(\Vert DX_{jk}\Vert_{\mathfrak{H}} ^{2}\right) + \E\left(X_{jk}^{2}\right) \E\left(\Vert DX_{ik}\Vert_{\mathfrak{H}} ^{2}\right)\right) =1
\end{eqnarray*}
and thus, writting $1=\frac{q_{i}}{q_{i}+q_{j}}+ \frac{q_{j}}{q_{i}+q_{j}}$,
 \begin{eqnarray*}
&&\left|  \langle D\widetilde{W}_{ij}, -DL^{-1} \widetilde{W}_{ij}\rangle_{\mathfrak{H}}-1\right|\\
&&\quad \leq \frac{1}{d \left(q_{i}+ q_{j}\right)}\left| \sum_{k=1}^{d} \left( X_{ik}^{2} \Vert DX_{jk}\Vert_{\mathfrak{H}} ^{2} -q_{j} \right)\right| + \frac{1}{d \left(q_{i}+ q_{j}\right)}\left| \sum_{k=1}^{d} \left( X_{jk}^{2} \Vert DX_{ik}\Vert_{\mathfrak{H}} ^{2} -q_{i} \right) \right|
\end{eqnarray*}
and
 \begin{eqnarray}
&& \E\left(\left|  \langle D\widetilde{W}_{ij}, -DL^{-1} \widetilde{W}_{ij}\rangle_{\mathfrak{H}}-1\right|^
   {2}\right) \leq  \frac{2}{d^2 \left(q_{i}+ q_{j}\right)^2} \E\left(\left| \sum_{k=1}^{d} \left( X_{ik}^{2} \Vert DX_{jk}\Vert_{\mathfrak{H}} ^{2} -q_{j} \right)\right| ^ {2}\right)\nonumber \\
&&\qquad\qquad\qquad\qquad\qquad +  \frac{2}{d^2 \left(q_{i}+ q_{j}\right)^2}\E\left(\left|
   \sum_{k=1}^{d} \left( X_{jk}^{2} \Vert DX_{ik}\Vert_{\mathfrak{H}} ^{2}
   -q_{i} \right)\right| ^ {2}\right).\label{28i-5}
\end{eqnarray}
The two summands above can be estimated in a similar way, so we only
cover the first one. By the independence of the entries and
\cite[Lemma 1]{bourguin_cramer_2011}, we have
\begin{eqnarray}
 \E\left(\left| \sum_{k=1}^{d} \left( X_{ik}^{2} \Vert
   DX_{jk}\Vert_{\mathfrak{H}} ^{2} -q_{j} \right)\right| ^ {2}\right)
   &\leq &  2 \E\left(\left| \sum_{k=1}^{d}  X_{ik}^{2} \left(\Vert
   DX_{jk}\Vert_{\mathfrak{H}} ^{2} -q_{j} \right)\right| ^ {2}\right)
   \nonumber \\
  && \qquad\qquad\qquad\quad +2 \E\left(\left| \sum_{k=1}^{d}  q_{j}\left( X_{ik}^{2} -1\right) \right| ^ {2}\right) \nonumber\\
&= & 2 \sum_{k=1}^{d}  \E\left( \left[X_{ik}^{2} \left(\Vert
     DX_{jk}\Vert_{\mathfrak{H}} ^{2} -q_{j} \right)\right] ^ {2}\right)
     \nonumber \\
   && \qquad\qquad\qquad\quad
      +2 \sum_{k=1}^{d}  q_{j}^ {2} \E\left(\left( X_{ik}^ {2}-1 \right) ^ {2}\right) \nonumber \\
&=& 2 \sum_{k=1}^{d}   \E\left(  \left(\Vert DX_{jk}\Vert_{\mathfrak{H}} ^{2} -q_{j}
    \right)^ {2}\right) m_4 \nonumber \\
  && + 2 d(m_{4}-1).\label{28i-6}
\end{eqnarray}
Writing
\begin{equation*}
\Vert D X_{jk}\Vert_{\mathfrak{H}}^{2}- q_{j}  = q_{j}^ {2} \sum_{r=0} ^
{q_{j}-2} r! \binom{q_j -1}{r}^2 I_{2q_{j}-2r-2}(f_{ik}\tilde{\otimes }_{r+1} f_{ik})
\end{equation*}
and estimating the $L^ {2}$-norm as in the proof of \eqref{28i-2} yields
\begin{equation}
\label{28i-7} \E\left( \left( \Vert D X_{jk}\Vert_{\mathfrak{H}}^{2}- q_{j} \right) ^ {2}\right) \leq  C(j),
\end{equation}
where $C(j)$ is a constant depending solely on $q_j$.
By \eqref{28i-5}, \eqref{28i-6} and \eqref{28i-7}, we get
\begin{equation}\label{bb2}
\E \left( \left(  \langle D\widetilde{W}_{ij}, -DL^{-1} \widetilde{W}_{ij}\rangle_{\mathfrak{H}} -1 \right) ^{2}\right) \leq  \frac{C(i,j)}{d},
\end{equation}
where $C(i,j)$ is a positive constant depending only on $q_i$ and $q_j$.
\\~\\
\noindent  {\bf Step 3: calculation of $ \E \left( \left(  \langle
      D\widetilde{W}_{ij}, -DL^{-1}
      \widetilde{W}_{ab}\rangle_{\mathfrak{H}}  \right) ^{2}\right)  $
  with $(i,j)\neq (a,b)$}.
\\~\\
Let $1\leq i,j, a, b\leq n$ with $i\leq j$, $a\leq b$ and $(i,j)\neq
(a,b)$. If $i,j, a,b$ are all distinct, then we have
\begin{eqnarray}
 \langle D\widetilde{W}_{ij}, -DL^{-1} \widetilde{W}_{ab}\rangle_{\mathfrak{H}} &=& \frac{1}{dq_{a}} \sum_{k,l=1}^ {d} \langle D(X_{ik}X_{jk}),  D(X_{al}X_{bl})\rangle_{\mathfrak{H}}\nonumber\\
&=&\frac{1}{dq_{a}} \sum_{k,l=1}^ {d} \langle X_{ik}DX_{jk}+
    X_{jk}DX_{ik}, X_{al}DX_{bl}+X_{bl} DX_{al}\rangle_{\mathfrak{H}} \nonumber \\
&=& \frac{1}{dq_{a}} \sum_{k,l=1}^ {d}\left( X_{ik}X_{al} \langle DX_{jk}, DX_{bl}\rangle_{\mathfrak{H}} +  X_{ik}X_{bl} \langle DX_{jk}, DX_{al}\rangle_{\mathfrak{H}} \right.\nonumber\\
&&\left.+ X_{jk}X_{al} \langle DX_{ik},
   DX_{bl}\rangle_{\mathfrak{H}} \right.\nonumber \\
  && \left. + X_{jk}X_{bl} \langle DX_{ik}, DX_{al}\rangle_{\mathfrak{H}}\right) =0,\label{28i-9}
\end{eqnarray}
since all the scalar products vanish according to Lemma \ref{ll3}.
\\~\\
The remaining cases, namely $ \langle D\widetilde{W}_{ij}, -DL^{-1}
\widetilde{W}_{ib}\rangle_{\mathfrak{H}}$ with $j\neq b$ and $ \langle
D\widetilde{W}_{ij}, -DL^{-1} \widetilde{W}_{aj}\rangle_{\mathfrak{H}}$
with $i\neq a$ can all be dealt with in a similar manner. For
instance, if $j\neq b$, assuming $i<j$ and $i<b$, we can write
\begin{eqnarray*}
\langle D\widetilde{W}_{ij}, -DL^{-1} \widetilde{W}_{ib}\rangle_{\mathfrak{H}} &=& \frac{1}{2dq_{i}}\sum_{k=1}^{d} \langle D(X_{ik}X_{jk}), D(X_{ik}X_{bk})\rangle_{\mathfrak{H}} \\
&=& \frac{1}{2dq_{i}}\sum_{k=1}^{d} X_{jk} X_{bk} \Vert DX_{ik}\Vert_{\mathfrak{H}}^{2}.
\end{eqnarray*}
Similarly, we also have
\begin{eqnarray}
\E \left(\left( \langle D\widetilde{W}_{ij}, -DL^{-1}
  \widetilde{W}_{aj}\rangle_{\mathfrak{H}} \right) ^{2}\right) = \frac{C(i)}{d^{2}} \sum_{k=1}^{d} \E\left(   \Vert DX_{ik}\Vert_{\mathfrak{H}}^{4}\right) \leq \frac{C(i)}{d},\label{bb3}
\end{eqnarray}
where the above equality and inequality are derived similarly as for
what was done for the bound appearing in \eqref{28i-7}.
\\~\\
An application of Lemma \ref{ll1} together with Theorem \ref{tt1} yields
\begin{equation*}
d_{W} (\widetilde{\W}_{n, d}, \Z_n )  \leq \sqrt{2}C \sqrt{
  \sum_{i,j,a,b=1}^ {n}\E\left( \left(  \langle D\widetilde{W}_{ij},
      -DL^{-1}\widetilde{W}_{ab} \rangle_{\mathfrak{H}} -\E\left( Z_{ij}Z_{ab}\right) \right)^ {2}\right)},
\end{equation*}
where $C>0$ is the constant appearing in Theorem \ref{tt1}. Since for $a,b,i,j$ all distinct, the corresponding above summands
vanish according to \eqref{28i-9}, we have only $n^ {4}- n(n-1)(n-2) (n-3)
\leq 6n^ {3}$ summands. By \eqref{bb1}, \eqref{bb2} and \eqref{bb3},
all these non-vanishing summands are bounded by $\frac{C}{d}$, where
$C>0$ denotes a generic constant resulting from the aggregation of the
$C(i)$ and $C(i,j)$ constants appearing in the previous steps of the proof. This
yields \eqref{28i-10} and concludes the proof. 
\end{proof}

\section{Random matrices with correlated second chaos entries}\label{sec4}

\noindent In this section, we consider the case where the entries of
the matrix $\mathcal{X}_{n,d}$ are allowed to be correlated. As in the
previous section, let $\mathcal{X}_{n,d} = \left( X_{ij} \right)_{1
  \leq i \leq n,\ 1 \leq j \leq d}$ be a $n \times d$ random matrix
whose entries are given by the increments of a Rosenblatt process,
which lives in the second Wiener chaos. The choice of dealing with the
second chaos in the case of correlated entries comes both from the
accrued importance of the second chaos in applications, as well as
from technical considerations of keeping the involved combinatorics at
a reasonable level for our exposition. The Rosenblatt process $(Z^
{H}_{t})_{t\geq 0}$ with self-similarity parameter $H\in
\left(\frac{1}{2}, 1\right)$  is defined by, for every $t \geq 0$,
\begin{equation}
\label{zh}
Z^ {H}_{t}= I_{2}(L_{t}),
\end{equation}
where $I_{2}$ denote the multiple Wiener integral of order two with
respect to a Brownian motion $(B_{t})_{t\in \mathbb{R}_{+}}$ and the
kernel $L_{t}$ is given by, for every $y_{1}, y_{2} \in \mathbb{R}$ and $t\geq 0$,
\begin{equation}
\label{L}
L_{t}(y_{1}, y_{2})= d(H) \mathds{1}_{[0,t]^2}(y_{1}, y_{2}) \int_{y_{1}\vee y_{2}}^ {t} \partial _{1} K^ {\frac{H+1}{2} }(u, y_{1})  \partial _{1} K^ {\frac{H+1}{2} }(u, y_{1})du,
\end{equation}
where
\begin{equation}
\label{dh}d(H)= \frac{1}{H+1} \sqrt{ \frac{2(2H-1)}{H}}
\end{equation}
and for $t>s$,
\begin{equation*}
K ^ {H}(t,s)= c(H) s ^ {\frac{1}{2}-H} \int_{s}^ {t} (u-s)^ { H-\frac{3}{2}} u ^ {H-\frac{1}{2}}du,
\end{equation*}
with $c(H)= \sqrt{\left( \frac{H(2H-1)}{\beta (2-2H, H-\frac{1}{2})}\right)}$, where $\beta$ denotes the beta function (see e.g. \cite{nualart_malliavin_2006}). 
\\~\\
The kernel $L_{t}$ belongs to $L^ {2}\left(\mathbb{R}^ {2}_{+}\right)$
for every $t\geq 0$. The Rosenblatt process $Z^ {H}$ is $H$-self
similar, has stationary increments and long memory. We refer to the
monographs \cite{pipiras_long-range_2017} or \cite{tudor_analysis_2013} for its basic properties. In
particular, it has the  same covariance as the fractional Brownian
motion, i.e., for any $s,t \geq 0$,
\begin{equation*}
\E( Z^ {H}_{t} Z^ {H}_{s})= \frac{1}{2}\left( t^ {2H}+ s^ {2H}-\vert t-s\vert ^ {2H}\right).
\end{equation*} 
A random variable with the same distribution as $Z^ {H}_{1}$ will be
called a {\it Rosenblatt random variable}. 
\\~\\
Let us now define the entries of the matrix $\mathcal{X}_{n,d}$. Let
$B= \left( B ^{1},\ldots, B ^{n}\right) $ denote a $d$-dimensional
Brownian motion and define
\begin{equation}
\label{zhi}
Z ^{H, i} _{t}= I^{i} _{2} (L_{t}), 
\end{equation}
where $I_{q}^{i}$ denotes the multiple Wiener integral of order $q$
with respect to the Brownian motion $B ^{i}$ for any $1 \leq i \leq
n$. Then, by \eqref{zh}, the processes $(Z ^{H,i}_{t})_{t\geq 0} $,
$1\leq i\leq n$  are independent Rosenblatt processes with the same
Hurst parameter (or self-similarity parameter) $H\in
\left(\frac{1}{2}, 1\right)$. For any $i \geq 1$, denote $f_{i}=L_{i}-L_{i-1}$. 
\\~\\
Consider the random matrix $\mathcal{X}_{n, d} = (X_{ij})_{1\leq i\leq
  n,\ 1\leq j\leq d}$ with entries given by
\begin{equation}\label{xij2}
X_{ij}= I _{2} ^{i }(f_{j})=Z ^ {H,i}_{j}- Z ^ {H,i}_{j-1} 
\end{equation}
for every $1\leq i\leq n$ and $1\leq j\leq d$, with $Z^ {H, i}$ given
by \eqref{zhi}. This means that all the entries have the same
distribution, the ones on different columns are independent and those
on the same rows are correlated according to the correlation structure
of the increments of the Rosenblatt process.  Since the covariance of
the Rosenblatt process coincides with that of the fractional Brownian
motion, the correlation structure of our matrix  is the same as in
\cite{nourdin_asymptotic_2018} (where the entries are given by the increments of the
fractional Brownian motion). Despite this fact, the non-Gaussian character will yield a different limiting behavior of the associated Wishart matrix.
\\~\\
More precisely, we have for every $1\leq 1,k\leq n$ and $1\leq j, l\leq d$,
\begin{equation*}
\E \left( X_{ij} X_{kl}\right) = \mathds{1}_{\left\{ i=k \right\}} \rho_{H} (j-l),
\end{equation*}
where $\rho_{H} $ denotes the correlation function of the Rosenblatt
process (or that of the fractional Brownian motion) given by, for $k
\in \mathbb{Z}$,
\begin{equation}\label{ro}
\rho_{H} (k)= \frac{1}{2}\left( \vert k+1\vert ^{2H}+ \vert k-1\vert ^{2H}- 2\vert k\vert ^{2H}\right).
\end{equation}
In particular, for $1\leq i\leq n$ and $1\leq j\leq d$,
\begin{equation*}
\E \left( X_{ij} ^{2}\right)= 2! \langle f_{i}, f_{j} \rangle_{L^2(\R_+^2)} =1.
\end{equation*}

\subsection{Rosenblatt limiting distribution}
Consider the Wishart matrix $\mathcal{W}_{n,d}$ obtained from
$\mathcal{X}_{n,d}$ as in \eqref{wishart}, where $\mathcal{X}_{n,d}$
is now given by \eqref{xij2}. Recall that the entries of the Wishart matrix are given by \eqref{wii} and \eqref{wij}.  We start by analyzing the asymptotic behavior in distribution of each element of the Wishart matrix. This will be related to the limiting behavior of the quadratic variations of the Rosenblatt process. Consider the constant $c_{1,H}$ given by
\begin{equation}
\label{c1h}
c_{1,H}= 4d(H),
\end{equation}
with $d(H)$ given by \eqref{dh}. Let us recall the following result from \cite{tudor_variations_2009}.
\begin{theorem}\label{tt3}
Let $(Z^ {H}_{t})_{t\geq 0}$ be a Rosenblatt process. Define, for $d\geq 1$,
\begin{equation}
\label{vd}
V_{d}= c_{1,H}^ {-1} d ^ {-H} \sum_{k=0} ^ {d-1} \left[ \frac{ \left(Z^ {H}_{\frac{k+1}{d}}-Z^ {H}_{\frac{k}{d}}\right) ^ {2}}{d^ {-2H}} -1\right]. 
\end{equation}
Then, the sequence $(V_{d})_{d\geq 1}$ converges in $L^ {2}(\Omega)$, as $d\to \infty$, to the Rosenblatt random variable $Z^ {H}_{1}$. 
\end{theorem}
\noindent Let us first  study the limiting behavior, as $d\to \infty$, of the diagonal terms of the Wishart matrix $\mathcal{W}_{n,d}$. 
\begin{prop}\label{pp1}
For $1\leq i\leq n$, let $W_{ii}$ be given by \eqref{wii}, and let
\begin{equation*}
\widetilde{W}_{ii}= c_{1,H} ^ {-1}d ^ {1-H} W_{ii},
\end{equation*}
where $c_{1,H}$ is the constant defined in \eqref{c1h}. Then, for every $1\leq i\leq n$,
\begin{equation*}
\widetilde{W}_{ii}\to Z ^ {H, i}_{1}
\end{equation*}
in $L ^ {2}(\Omega)$ as $d \to \infty$.
\end{prop}
\begin{proof}
By the scaling property of the Rosenblatt process and \eqref{xij2}, we
have, for every $1\leq i\leq n$, 
\begin{eqnarray*}
W_{ii}&=&\frac{1}{d} \sum_{k=1}^ {d} (X_{ik} ^ {2}-1) =\frac{1}{d}
          \sum_{k=1}^ {d} \left( \left( Z^ {H,i}_{k+1}- Z ^
          {H,i}_{k}\right)^ {2}-1\right)      \\
&\stackrel{\mathscr{D}}{=} &\frac{1}{d}\sum_{k=0} ^ {d-1} \left( \frac{ \left(Z^ {H,i}_{\frac{k+1}{d}}-Z^ {H,i}_{\frac{k}{d}}\right) ^ {2}}{d^ {-2H}} -1\right)=c_{1,H}d^ {H-1} V_{d} ^ {i},
\end{eqnarray*} 
where $\stackrel{\mathscr{D}}{=} $ denotes equality in distribution, and for $1\leq i\leq n$,
\begin{equation}\label{vdi}
V_{d} ^ {i} = c_{1,H}^ {-1} d ^ {-H} \sum_{k=0} ^ {d-1} \left( \frac{ \left(Z^ {H,i}_{\frac{k+1}{d}}-Z^ {H,i}_{\frac{k}{d}}\right) ^ {2}}{d^ {-2H}} -1\right).
\end{equation}
The conclusion follows from Theorem \ref{tt3}. 
\end{proof}
\noindent As far as the convergence of the non-diagonal terms of the Wishart matrix \eqref{wishart}, we have the following result. It shows that the square mean of the non-diagonal terms of the renormalized Wishart matrix is dominated by the square mean of the diagonal terms. Intuitively, this happens because the mean square of the  non-diagonal terms involves the increments  of two independent Rosenblatt processes.  
\begin{prop}\label{pp2}
For $1\leq i,j\leq n$ with $i\neq j$, let $W_{ij}$ be given by \eqref{wij}, and define
\begin{equation}
\label{twij}
\widetilde{W}_{ij}= c_{1,H}^{-1}d^{1-H} W_{ij},
\end{equation}
where $c_{1,H}$ denotes the constant defined in \eqref{c1h}. Then, for every $1\leq i,j\leq n$,
\begin{equation*}
\widetilde{W}_{i,j}\to 0
\end{equation*}
in $L^ {2}(\Omega)$ as $d \to \infty$, and
\begin{equation}\label{31i-1}
\E\left( \widetilde{W}_{ij} ^ {2}\right) \leq C \begin{cases} d ^
  {1-2H} & \mbox{if } H \in \left(\frac{1}{2}, \frac{3}{4}\right) \\
\log(d) d ^ {-\frac{1}{2}}& \mbox{if } H=\frac{3}{4}\\
 d ^ {2H-2}& \mbox{if } H\in \left( \frac{3}{4}, 1\right)
\end{cases},
\end{equation}
where $C>0$ denotes a generic constant.
\end{prop}
\begin{proof}
By self-similarity and \eqref{xij2},
\begin{eqnarray*}
W_{ij}& =&\frac{1}{d} \sum_{k=1}^ {d} X_{ik}X_{jk}=\frac{1}{d} \sum_{k=0}^ {d-1} \left( Z ^ {H,i}_{k+1}- Z^ {H,i}_{k}\right)  \left( Z ^ {H,j}_{k+1}- Z^ {H,j}_{k}\right),
\end{eqnarray*}
so that, for every $1\leq i,j\leq n$, 
\begin{eqnarray*}
\E\left(\widetilde{W}_{ij}^ {2}\right)&=& c_{1,H}^{-2} d^ {-2H} \E\left( \left(  \sum_{k=0}^ {d-1} \left( Z ^ {H,i}_{k+1}- Z^ {H,i}_{k}\right)  \left( Z ^ {H,j}_{k+1}- Z^ {H,j}_{k}\right)\right) ^ {2}\right) \\
&=& c_{1,H}^{-2} d^ {-2H} \sum_{k,l=0} ^{d-1} \E\left(   \left( Z ^
    {H,i}_{k+1}- Z^ {H,i}_{k}\right) \left( Z ^ {H,i}_{l+1}- Z^
    {H,i}_{l}\right)\right) \\
  && \qquad\qquad\qquad\qquad\qquad\qquad\qquad\qquad \E \left(   \left( Z ^ {H,j}_{k+1}- Z^ {H,j}_{k}\right) \left( Z ^ {H,j}_{l+1}- Z^ {H,j}_{l}\right)\right)\\
&=& c_{1,H}^{-2} d^ {-2H} \sum_{k,l=0} ^{d-1}\rho_{H}(\vert k-l\vert ) ^ {2}\\
&\leq &c_{1,H}^{-2} d ^ {1-2H} \sum_{v\in \mathbb{Z}} \rho_{H}(\vert v\vert) ^{2} \left( 1-\frac{\vert v\vert}{n}\right)1_{(\vert v\vert < n)},
\end{eqnarray*}
where $\rho_{H}$ is given by \eqref{ro}. The fact that $\rho_{H}(\vert k\vert ) $
behaves as $H(2H-1)\vert k\vert ^ {2H-2} $ as $\vert k\vert \to
\infty$ concludes the proof. 
\end{proof}

\subsection{Proof of Theorem \ref{mainresult-correlated}}

In this section, we pave the way to the proof of Theorem
\ref{mainresult-correlated} by stating and proving some preparatory
results, making use of the results established in the previous
subsection to do so. Theorem
\ref{mainresult-correlated} is restated for convenience at the end of
the section right before its proof.
\\~\\
Consider the renormalized Wishart matrix $\widetilde{\W}_{n, d}$
defined in \eqref{twij}. By Propositions \ref{pp1} and \ref{pp2}, its
limit in distribution is an $n\times n$  diagonal matrix, denoted by
$\mathcal{R}_{n}^ {H}= (R_{ij}^ {H})_{1\leq i,j\leq n}$, with
independent diagonal entries given by, for all $1\leq i\leq n$,
\begin{equation}
\label{29i-1}
R^ {H}_{ii}= Z ^ {H, i}_{1}.
\end{equation}
Given that what we need is to estimate the Wasserstein distance between
$\widetilde{\W}_{n, d}$ and $\mathcal{R}_{n}^ {H}$, we start with the
observation that, due to the scaling property of the Rosenblatt
process, we have
\begin{equation*}
d_{W}\left(\widetilde{\W}_{n,d}, \mathcal{R}_{n}^ {H}\right)=
d_{W}\left(\mathcal{V}_{n,d}, \mathcal{R}_{n}^ {H}\right),
\end{equation*}
where the matrix $\mathcal{V}_{n,d}=(V_{ij})_{1\leq i, j\leq n} $ is given by 
\begin{equation*}
  \begin{cases}
  \displaystyle   V_{ii}= V_{d}^ {i} & \mbox{for } 1 \leq i \leq n\\
  \displaystyle  V_{ij} =c_{1,H}^ {-1} d ^ {H}\sum_{k=0} ^ {d-1} \left( Z ^
      {H,i}_{\frac{k+1}{d}}-   Z ^ {H,i}_{\frac{k}{d}}\right)  \left(
      Z ^ {H,j}_{\frac{k+1}{d}}-   Z ^ {H,j}_{\frac{k}{d}}\right) &
    \mbox{for } 1 \leq i\neq j \leq n
  \end{cases},
\end{equation*}
where $V_{d}^ {i}$ was defined in \eqref{vdi}. By the definition of
the Wasserstein distance \eqref{dw}, 
\begin{equation}\label{29i-2}
d_{W}\left(\widetilde{\W}_{n,d}, \mathcal{R}_{n}^ {H}\right)=
d_{W}\left(\mathcal{V}_{n,d}, \mathcal{R}_{n}^ {H}\right) \leq \sqrt{
  \sum_{i,j=1}^ {n}\E \left( \left( V_{ij} - R^ {H}_{ij}\right) ^ {2}\right) }
\end{equation}
with $ R^ {H} _{ij}=0 $ if $i\neq j$ and $R^ {H}_{ii}$ given by \eqref{29i-1}. 
\\~\\
The estimates for the terms with $i\neq j$ in the right-hand side of
\eqref{29i-2} will follow from Proposition \ref{pp2}. The next
proposition provides estimates for the diagonal summands of the right-hand side of \eqref{29i-2}. 
\begin{prop}\label{pp3}
Let $V_{d}$ be given by \eqref{vd}. Then, it holds that
\begin{equation}\label{30i-6}
\E\left(\left| V_{d} - Z ^ {H}_{1} \right| ^ {2}\right) \leq
C \begin{cases} d ^ {1-2H} & \mbox{if } H \in\left( \frac{1}{2}, \frac{3}{4}\right)\\
\log(d) d ^ {-\frac{1}{2}} & \mbox {if } H=\frac{3}{4}\\
d ^ {2H-2} &\mbox{if } H\in \left( \frac{3}{4}, 1\right)
\end{cases},
\end{equation}
where $C >0$ denotes a generic constant.
\end{prop}
\begin{proof}
For $0\leq k\leq d-1$, we have
$$ Z ^ {H}_{\frac{k+1}{d}}-   Z ^ {H}_{\frac{k}{d}}= I_{2} \left(L_{\frac{k+1}{n}}- L_{\frac{k}{n}} \right),$$
where $L$ is the kernel defined in \eqref{L}. By the product formula
for multiple Wiener integrals \eqref{prod}, we can decompose $V_{d}$
as the sum of two terms, one in the fourth Wiener chaos and one in the
second Wiener chaos. Namely,
\begin{eqnarray}
V_{d}&=&c_{1,H}^ {-1} d ^ {H}\sum_{k=0} ^ {d-1}\left[ I_{4}\left( \left(  L_{\frac{k+1}{n}}- L_{\frac{k}{n}} \right)^ {\otimes 2}\right)+ 4I_{2}  \left(  \left(L_{\frac{k+1}{n}}- L_{\frac{k}{n}}\right)\otimes_{1} \left(L_{\frac{k+1}{n}}- L_{\frac{k}{n}}\right)\right)\right]\nonumber\\
&=& T_{4, d}+ T_{2,d}.\label{t2t4}
\end{eqnarray}
The estimation of the $L^ {2}(\Omega)$-norm of the term $T_{4, d}$ has
been done in \cite{tudor_variations_2009}. This term has no contribution to the limit of
$V_{d}$ and using \cite[Equations (3.15)--(3.17)]{tudor_variations_2009} yields
\begin{equation*}
\E\left( T_{4,d}^{2}\right) \leq C\begin{cases} d ^ {1-2H} & \mbox{if } H \in\left( \frac{1}{2}, \frac{3}{4}\right)\\
\log(d) d ^ {-\frac{1}{2}} & \mbox {if } H=\frac{3}{4}\\
d ^ {2H-2} & \mbox{if } H\in \left( \frac{3}{4}, 1\right)
\end{cases}.
\end{equation*} 
The summand $T_{2,d}$ appearing in \eqref{t2t4} converges in $L^
{2}(\Omega)$ to $Z^ {H}_{1}$. This has also been proved in \cite{tudor_variations_2009},
but we still need to evaluate the rate of this convergence.  We can
write $T_{2, d}= I_{2}(h_{d}) $, with 
\begin{equation}
h_{d}(y_{1}, y_{2}) =4 c_{1, H}^ {-1} d ^ {H}\sum_{k=0}^{d-1} \left(  \left(L_{\frac{k+1}{n}}- L_{\frac{k}{n}}\right)\otimes_{1} \left(L_{\frac{k+1}{n}}- L_{\frac{k}{n}}\right)\right).\label{hd}
\end{equation}
Hence,
\begin{equation}\label{30i-1}
\E\left(\left| T_{2, d}- Z ^ {H}_{1} \right| ^ {2}\right)=\E\left( \left| T_{2, d}\right| ^ {2}\right) -2\E \left(T_{2, d} Z ^ {H}_{1} \right) + \E\left(\left| Z^ {H}_{1}\right| ^ {2}\right).
\end{equation}
On one hand, \cite[Equation (3.11)]{tudor_variations_2009} yields
\begin{eqnarray*}
\E\left(\left| T_{2, d} \right| ^ {2}\right)&=&2\Vert h_{d}\Vert_{L^2(\R_{+}^2)} ^{2} \\
&=&   2c_{1,H}^ {-2} d(H)^ {4} (H( H+1)) ^ {4} d ^{2H}\sum_{i, j=0} ^{d-1}  \int_{\frac{i}{d}} ^{\frac{i+1}{d}} \int_{\frac{i}{d}} ^{\frac{i+1}{d}}  \int_{\frac{j}{d}} ^{\frac{j+1}{d}} \int_{\frac{j}{d}} ^{\frac{j+1}{d}} \vert u-v\vert ^{H-1} \\
&&\qquad\qquad\qquad\qquad \vert u'-v'\vert ^{H-1} \vert u-u'\vert ^{H-1} \vert v-v'\vert ^{H-1}dudvdu'dv' \\
&=&H(2H-1)e(H) d ^{-2H} \sum_{i, j=0} ^{d-1}  \int_{[0, 1] ^{4}}
    \vert u-v\vert ^{H-1} \vert u'-v'\vert ^{H-1}  \\
  && \qquad\qquad\qquad\qquad \vert u-u'+i-j\vert^{H-1} \vert v-v'+i-j\vert ^{H-1}dudvdu'dv',
\end{eqnarray*}
where $e(H)$ is a constant given by
\begin{equation}\label{eh}
e(H)= \frac{H^ {2}(H+1) ^ {2}}{4} .  
\end{equation}
On the other hand, using the fact that $2\Vert L_{1}\Vert_{L^2(\R_{+}^2)} ^{2}= 1$ yields
\begin{eqnarray*}
\E\left(\left| Z^ {H}_{1}\right| ^ {2}\right)&=&2\Vert L_{1}\Vert_{L^2(\R_+^2)} ^{2}\\
                                             &=& H(2H-1) \int_{0}^{1}  \int_{0} ^{1}  \vert u-v\vert ^{2H-2}dudv \\
  &=& H(2H-1) \sum_{i, j=0} ^{d-1} \int_{\frac{i}{d}} ^{\frac{i+1}{d}} \int_{\frac{j}{d}} ^{\frac{j+1}{d}}  \vert u-v\vert ^{2H-2}dudv\\
&=& H(2H-1)d ^ {-2H} \sum_{i, j=0} ^{d-1}\int_{[0, 1]^{2}}  \vert
   u-v+i-j\vert ^{2H-2}dudv \\
  &=& 1.
\end{eqnarray*}
Furthermore, note that \eqref{hd} and \eqref{L} imply
\begin{eqnarray*}
\E \left(T_{2, d} Z ^ {H}_{1} \right)&=&2\langle h_{N}, L_{1} \rangle_{L^2(\R_+^2)} \\
&=& H(2H-1)f(H)d ^{H}\sum_{i=0} ^{d-1} \int_{\frac{i}{d}}
    ^{\frac{i+1}{d}} \int_{\frac{i}{d}} ^{\frac{i+1}{d}} \int_{0}^{1}
    \vert u-v\vert ^{H-1} \vert u-u'\vert^{H-1}\\
  && \qquad\qquad\qquad\qquad\qquad\qquad\qquad\qquad\qquad\vert v-u'\vert ^{H-1}dudvdu'  \\
&=&H(2H-1)f(H) d ^{H}  \sum_{i,j=0} ^{d-1} \int_{\frac{i}{d}}
    ^{\frac{i+1}{d}} \int_{\frac{i}{d}} ^{\frac{i+1}{d}}
    \int_{\frac{j}{d}} ^{\frac{j+1}{d}} \vert u-v\vert ^{H-1} \vert
    u-u'\vert ^{H-1}\\
  &&\qquad\qquad\qquad\qquad\qquad\qquad\qquad\qquad\qquad \vert v-u'\vert ^{H-1}dudvdu'  \\
&=&H(2H-1) f(H)d ^{-2H} \sum_{i,j=0} ^{d-1}\int_{[0, 1] ^{3}} \vert
    u-v\vert ^{H-1} \vert u-u'+i-j\vert ^{H-1} \\
  &&\qquad\qquad\qquad\qquad\qquad\qquad\qquad\qquad \vert v-u'+i-j\vert ^{H-1}dudvdu',
\end{eqnarray*}
where $f(H)$ is a constant given by
\begin{equation}\label{fh}
f(H)=\frac{H+1}{2(2H-1)}.
\end{equation}
Now, \eqref{30i-1} becomes
\begin{eqnarray}
&&\E\left(\left| T_{2, d}- Z ^ {H}_{1} \right| ^ {2}\right)\nonumber \\
&&\qquad = H(2H-1) d ^ {-2H}e(H) \sum_{i,j=0}^ {d-1} \left[\int_{[0, 1] ^{4}}
    \vert u-v\vert ^{H-1} \vert u'-v'\vert ^{H-1} \vert u-u'+i-j\vert
    ^{H-1} \right. \nonumber\\
  && \left. \qquad\qquad\qquad\qquad\qquad\qquad\qquad\qquad\qquad\qquad \vert v-v'+i-j\vert ^{H-1}dudvdu'dv'\right.\nonumber \\
&&\left.\qquad\quad  -2f(H) \int_{[0, 1] ^{3}} \vert u-v\vert ^{H-1} \vert
   u-u'+i-j\vert ^{H-1}\vert v-u'+i-j\vert ^{H-1}dudvdu' \right. \nonumber\\
  && \left. \qquad\quad + \int_{[0, 1]^{2}}  \vert u-v+i-j\vert ^{2H-2}dudv\right]\nonumber \\
&&\qquad\leq  C d ^{1-2H} e(H) \sum_{k\in \mathbb{Z}} \left[\int_{[0,
   1] ^{4}}  \vert u-v\vert ^{H-1} \vert u'-v'\vert ^{H-1} \vert
   u-u'+k\vert ^{H-1} \right. \nonumber \\
  &&
     \left. \qquad\qquad\qquad\qquad\qquad\qquad\qquad\qquad\qquad\qquad\qquad \vert v-v'+k\vert ^{H-1}dudvdu'dv'\right. \nonumber  \\
&&\left. \qquad\quad -2 f(H)\int_{[0, 1] ^{3}} \vert u-v\vert ^{H-1} \vert
   u-u'+k\vert ^{H-1}\vert v-u'+k\vert ^{H-1}dudvdu' \right. \nonumber\\
  && \left. \qquad\quad + \int_{[0, 1]^{2}}  \vert u-v+k\vert ^{2H-2}dudv\right].\label{30i-2}
\end{eqnarray}
Now, \cite[Lemma 5]{tudor_variations_2009} (see also \cite[Lemma 2]{clausel_asymptotic_2014}) together with
the definition of $e(H)$ given in \eqref{eh} yields 
\begin{eqnarray}\label{30i-3}
 && \int_{[0, 1] ^{4}}  \vert u-v\vert ^{H-1} \vert u'-v'\vert
  ^{H-1} \vert u-u'+k\vert ^{H-1} \vert v-v'+k\vert ^{H-1} dudvdu'dv'
                                                  \nonumber
  \\
  &&\qquad\qquad\qquad\qquad\qquad\qquad\qquad\qquad
     = e(H)^ {-1}  k ^{2H-2} +O (k^{2H-2}).
\end{eqnarray}
Similarly,
\begin{equation}\label{30i-4}
\int_{[0, 1] ^{3}} \vert u-v\vert ^{H-1} \vert u-u'+k\vert ^{H-1}\vert
v-u'+k\vert ^{H-1}dudvdu' = f(H)^ {-1} k ^{2H-2} + o(k ^{2H-2}),
\end{equation}
where $f(H)$ is given by \eqref{fh}. Finally, \cite[Proof of
Proposition 3.1]{breton_error_2008} yields
\begin{equation}
\label{30i-5}\int_{[0, 1]^{2}}  \vert u-v+k\vert ^{2H-2}dudv = k^{2H-2} + o(k ^{2H-2}).
\end{equation}
Combining \eqref{30i-3}, \eqref{30i-4} and \eqref{30i-5} implies that
the sum over $k\in \mathbb{Z}$ in \eqref{30i-2}  converges. Hence,
\begin{equation*}
\E\left(\left| T_{2, d}- Z^ {H}_{1}\right| ^ {2}\right) \leq C d ^ {1-2H},
\end{equation*}
and since, by \eqref{t2t4},
\begin{equation*}
\E\left(\left| V_{d}- Z^ {H}_{1}\right| ^ {2}\right) =\E\left( \vert T_{4, d}\vert ^ {2}\right) + \E\left(\left| T_{2, d}- Z^ {H}_{1}\right|^ {2}\right),
\end{equation*}
we obtain \eqref{30i-6}.   
\end{proof}
We are now ready to provide the proof of Theorem
\ref{mainresult-correlated}, which we restate here
for convenience. 
\begin{customthm}{2}
Let $\widetilde{\W}_{n, d}$ be the renormalized Wishart matrix
\eqref{twij} and let $\mathcal{R}^ {H}_{n}$ be the diagonal matrix with
entries given by \eqref{29i-1}. Then, for every $n\geq 1$, the random  matrix  $\widetilde{\W}_{n, d}$ converges componentwise in distribution, as $d\to \infty$, to the matrix $\mathcal{R}^ {H}_{n}$. Moreover, as  $n,d\geq
1$, there exists a positive constant $C$ such that
\begin{equation*}
d_{W}\left( \widetilde{\W}_{n, d}, \mathcal{R}^ {H}_{n}\right) \leq
C \begin{cases} nd ^ {\frac{1}{2}-H} & \mbox{if } H \in\left( \frac{1}{2}, \frac{3}{4}\right)\\
n\sqrt{\log(d)} d ^ {-\frac{1}{4}} & \mbox {if } H=\frac{3}{4}\\
nd ^ {H-1} & \mbox{if } H\in \left( \frac{3}{4}, 1\right)
\end{cases}.
\end{equation*}
\end{customthm}
\begin{proof}[Proof of Theorem \ref{mainresult-correlated}]
The conclusion follows from combining relation \eqref{29i-2} with Propositions \ref{pp2} and \ref{pp3}.  Indeed, the summands with
$i\neq j$ in \eqref{29i-2} have been estimated in Proposition
\ref{pp2} (see \eqref{31i-1}), while the diagonal terms of
\eqref{29i-2} are estimated by \eqref{30i-6}.
\end{proof}

\bibliography{/home/solesne/Dropbox/work/research/biblio} 
\end{document}